\documentclass[a4paper,12pt]{amsart}
\usepackage{amsmath,amsthm,mathrsfs,mdframed}
\usepackage{amssymb,wasysym,xcolor}
\usepackage{amscd}
\usepackage[all]{xy}
\usepackage{color}
\usepackage{combelow}
\usepackage{enumerate}
\usepackage{hyperref}
\usepackage[all]{xy}
 \usepackage{geometry}
\newtheorem{theorem}{Theorem}

\newtheorem{corollary}[theorem]{Corollary}
\newtheorem{definition}[theorem]{Definition}
\newtheorem{example}[theorem]{Example}
\newtheorem{lemma}[theorem]{Lemma}
\newtheorem*{theorem*}{Theorem}

\newtheorem{proposition}[theorem]{Proposition}
\newtheorem{remark}[theorem]{Remark}

\newcommand{\diffto}{\xrightarrow{\raisebox{-0.2 em}[0pt][0pt]{\smash{\ensuremath{\sim}}}}}

\newcommand{\rmap}{\longrightarrow}
\newcommand{\lmap}{\longleftarrow}

\hypersetup{
        colorlinks=true, 
        linkcolor=blue,  
        citecolor=red, 
        urlcolor=cyan,   
        filecolor=magenta 
    }

\begin{document}
\title{Concurring reduction schemes for Dirac structures}
\author{Dan Aguero, Alessandro Arsie, \\Pedro Frejlich, and Igor Mencattini}
\maketitle

\begin{abstract}
The notion of \emph{concurrence} was recently proposed as the natural compatibility relation between Dirac structures, generalizing the commutativity of two Poisson structures. We address the question of when a reduction scheme --- that is, a way to induce a Dirac structure on a quotient of a submanifold --- respects this relation. After characterizing the minimal scheme of \emph{Dirac reduction}, we prove that two concurring Dirac structures have concurring reductions whenever they share a common \emph{witness}, extending to Dirac geometry the reduction of the Marsden-Ra\cb{t}iu theorem. 

Two procedures for constructing such common witnesses are given, the second being the Dirac counterpart of Magri's original recipe in bihamiltonian geometry. Examples drawn from Hamiltonian actions, Dirac-Nijenhuis manifolds, and complex Dirac structures conclude the paper and illustrate our methods.
\end{abstract}

\tableofcontents
\setcounter{tocdepth}{1}

\section{Introduction}

Noether's principle states that continuous symmetries of a system have a corresponding conservation law. The knowledge that a certain quantity is conserved along trajectories of the system often allows one to ``reduce the degrees of freedom'' of the system, in the sense that the description of the trajectories of the system can be described in terms of a smaller space --- namely, the orbits among those points in the phase space which share the same conserved quantities as the initial condition.

In Symplectic Geometry, the principle above takes the form of the theorem of Marsden-Weinstein-Meyer \cite{Meyer,MW} that, if a symplectic manifold has a Hamiltonian action of a Lie group $G$,
\begin{align*}
 G \curvearrowright (M,\omega_M) \stackrel{\mu}{\rmap} \mathfrak{g}^*,
\end{align*}
then\footnote{``Under favorable circumstances'' or ``modulo technical details'': for the sake of readability, we will not emphasize in this introduction technical details (typically demanding constant rank or connectedness conditions), and we refer the reader to the body of the text for complete statements.} the symplectic structure $\omega_M$ on $M$ \emph{reduces} to a symplectic form on $Y$,
\begin{align}\label{eq : reduction form}
 i^*(\omega_M) = p^*(\omega_Y),
\end{align}
where
\begin{align*}
 M \stackrel{i}{\lmap} X:=\mu^{-1}(c) \stackrel{p}{\rmap} Y:=X/G_c
\end{align*}
and $i:X \to M$ denotes the inclusion, $G_c$ the isotropy subgroup of $c$, and $p:X \to Y$ the quotient map. In the general Poisson setting, in which $\omega_M$ is replaced by a Poisson structure $\pi_M$, the analogue of the reduction relation \eqref{eq : reduction form} above can be written as follows: there is a Poisson structure $\pi_Y$ on $Y$, uniquely defined by the condition that, whenever we are given functions $f,g \in C^{\infty}(Y)$ and $F,G \in C^{\infty}(M)$ with $i^*(F) = p^*(f)$ and $i^*(G)=p^*(g)$, then
\begin{align}\label{eq : reduction bivector 1}
 i^*\{F,G\}^{\pi_M} = p^*\{f,g\}^{\pi_Y}.
\end{align}
Condition \eqref{eq : reduction bivector 1} can be written more succinctly in the language of \emph{Dirac geometry} \cite{Courant,Bursztyn,Meinrenken_perspective} --- a generalization of Poisson geometry in which it is meaningful to consider the pullback of a Poisson structure by a smooth map. Briefly, a Dirac structure $L_M$ on $M$ is a vector subbundle $L_M \subset \mathbb{T}M:=TM \oplus T^*M$ of the generalized tangent bundle $\mathbb{T}M$ of $M$, which is isotropic for the canonical symmetric pairing
\begin{align*}
 & \langle u+\xi,v+\eta\rangle:=\langle u,\eta\rangle+\langle \xi,v\rangle,
\end{align*}
and whose space of sections $\Gamma(L_M)$ is involutive under the \emph{Dorfman bracket}
\begin{align}\label{eq : Dorfman}
 [u+\xi,v+\eta]=[u,v]+\mathscr{L}_u\eta-\iota_v\mathrm{d}\xi.
\end{align}
The prototypical examples are the graphs
\begin{align*}
 && \mathrm{Gr}(\pi) = (\pi+\mathrm{id}_{T^*M})(T^*M),
 && \mathrm{Gr}(F) = F \oplus F^{\circ},
 && \mathrm{Gr}(\omega) = (\mathrm{id}+\omega)(TM)
\end{align*}
of a Poisson structure $\pi \in \Gamma(\wedge^2TM)$, a foliation $F \subset TM$ and a closed two-form $\omega \in \Gamma(\wedge^2T^*M)$. Under favorable circumstances, a Dirac structure $L_M$ on $M$ may be pulled back by a smooth map into $M$, or pushed forward by a map from $M$, and condition \eqref{eq : reduction bivector 1}
\begin{align*}
 i^!\mathrm{Gr}(\pi_M) = p^!\mathrm{Gr}(\pi_Y)
\end{align*}
just says that the pullbacks of $\pi_M$ and $\pi_Y$ to $X$ coincide \emph{as Dirac structures}.

This paper revolves around a few basic questions. The first concerns characterizing when a Dirac structure $L_M$ on $M$ can be transported by a diagram
\begin{align*}
 M \stackrel{i}{\lmap} X \stackrel{p}{\rmap} Y
\end{align*}
where $i$ is an injective immersion and $p$ is a smooth surjective submersion, to a Dirac structure on $Y$:

\begin{theorem*}[Dirac reduction]
 If $p$ has connected fibres, $p_!i^!(L_M)$ is a Dirac structure $L_Y$ on $Y$ exactly when $L_M[I]$ is Dirac along $i$, where
 \begin{align*}
  && L_M[I]=L_M \cap I^{\perp}+I,
  && I:=F\oplus N^*X,
  && F=\ker(p_*).
 \end{align*}
\end{theorem*}
When the conclusions of the above theorem are satisfied, we say that $L_M$ has a \emph{Dirac reduction} to $L_Y$ on $Y$. In the extreme case where $X=M$ and $i=\mathrm{id}_M$, Dirac reduction becomes the Dirac version of Libermann's theorem \cite{Libermann}. In the other extreme, when $X=Y$ and $p=\mathrm{id}_X$, it refines a known criterion for the pullback Dirac structure $i^!(L_M)$ to be defined.

To state the second basic question we address, let us recall that there is a natural compatibility condition between two Poisson structures, which demands that their sum be again a Poisson structure. This compatibility condition can be expressed for general Dirac structures $L_M$ and $R_M$ on $M$ by demanding that their cotangent product
\begin{align*}
 L_M \circledast R_M = \{ u_L + u_R + \xi \ | \ u_L + \xi \in L_M, \ u_R+\xi \in R_M\}
\end{align*}
be again a Dirac structure, in which case we say that $L_M$ and $R_M$ \emph{concur}. This notion was systematically developed in \cite{FMT_Concurrence}, where it is argued to be the natural compatibility relation between Dirac structures, and where its main properties are established. 

Regrettably, Dirac reduction does \emph{not} respect concurrence, in the sense that it is perfectly possible that the Dirac reductions to $Y$ do not concur, even if the original Dirac structures on $M$ concur. However, a reduction scheme of a different nature had been proposed by Marsden and Ra\cb{t}iu in \cite{MR} in the context of Poisson manifolds $(M,\pi_M)$. It equips the submanifold $X$ with a vector subbundle $E \subset TM|_X$ for which $F = E \cap TX$, and demands that $E^{\circ}$ be a Lie subalgebroid of $(T^*M,[\cdot,\cdot]^{\pi_M})$. This ensures a Dirac reduction of $\pi_M$ to a Poisson structure $\pi_Y$ --- we think of $E$ as \emph{witnessing} said Dirac reduction. As was later observed by \cite{CMP,Costa_Marle,CFMP}, Marsden-Ra\cb{t}iu reduction \emph{does} preserves concurrence, and that inspired us to extend the condition above on the ``witness'' $E$ of the Dirac reduction of $\pi_M$ to $Y$ to one that makes sense for all Dirac structures, which is accomplished by Definition \ref{def : witness}. This extension allows us to state the following concurring reduction scheme generalizing that of Marsden-Ra\cb{t}iu\footnote{One wonders if this is what Liu, Weinstein and Xu had in mind when they commented that ``it would be interesting to interpret and extend the general Marsden--Ra\cb{t}iu reduction in terms of Dirac geometry'' after \cite[Lemma 7.3]{Liu-Weinstein-Xu}.} :

\begin{theorem*}[Concurring reduction]
 If concurring Dirac structures $L_M$ and $R_M$ on $M$ have Dirac reductions $L_Y$ and $R_Y$ on $Y$, then also $L_Y$ and $R_Y$ concur, provided that their Dirac reductions have a common witness. 
\end{theorem*}

In the second part of the paper, we turn to the question of whether there exist ``recipes'' to construct a common witness for concurring Dirac structures $L_M$ and $R_M$. The prototypical recipe is what we will nickname

\begin{quotation}
{\bf Magri's original recipe \cite{CMP,CFMP}}   Suppose $\pi_L$ and $\pi_R$ are two commuting Poisson structures on $M$ and $X \subset M$ is a leaf of $\pi_R$. Then (under favorable circumstances) $E_L:=\pi_L(\ker \pi_R) \subset TM|_X$ is a witness for both $\pi_L$ and $\pi_R$, and the reduced Poisson structures $\nu_L$ and $\nu_R$ on $Y$ also commute.
\end{quotation}

The two main results in this direction are as follows. 

\begin{theorem*}[Reduction by kernels]
 Let $L_M$ and $R_M$ be concurring Dirac structures,  and let $K(L_M):=L\cap TM$ be the kernel of $L_M$ and similarly for $K(R_M).$ If the diagram on the left
 \begin{align*}
  & 
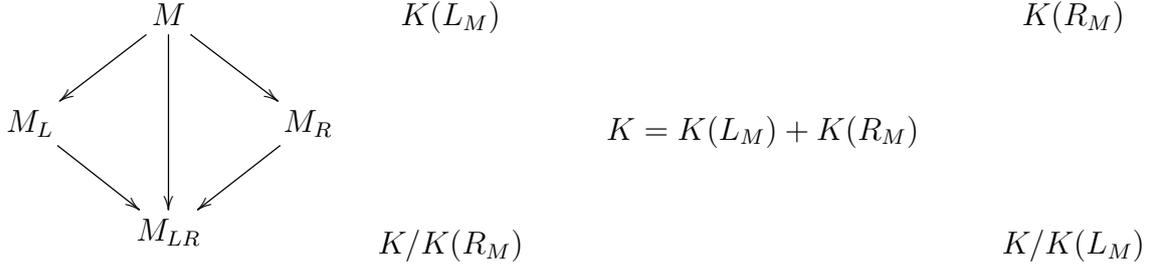
\xymatrix{
 & M \ar[dl] \ar[dr] \ar[dd] & \\
 M_L \ar[dr] & & M_R \ar[dl]\\
 & M_{LR}
 }
 & \xymatrix{
 K(L_M)  & & K(R_M) \\
 & K=K(L_M)+K(R_M) &\\
 K/K(R_M) & & K/K(L_M)
 }
\end{align*}
are the leaf spaces of the foliations on the right, then $L_M$ and $R_M$ push forward to concurring Dirac structures at every vertex, and those on the bottom vertex are commuting Poisson structures.
\end{theorem*}

The last main theorem gives a generalization of Magri's original recipe to the Dirac context:

\begin{theorem*}[Magri's recipe]
 Let $L_M$ and $R_M$ be concurring Dirac structures. Denote by $E_L$ and by $E_R$ the kernel of the Dirac structures
 \begin{align*}
  & L_M \star \left(R_M \circledast \mathcal{R}_{-1}(L_M)\right),
  & R_M \star \left(L_M \circledast \mathcal{R}_{-1}(R_M)\right),
 \end{align*}
 and suppose that maps in the diagram on the left
 \begin{align*}
  & 
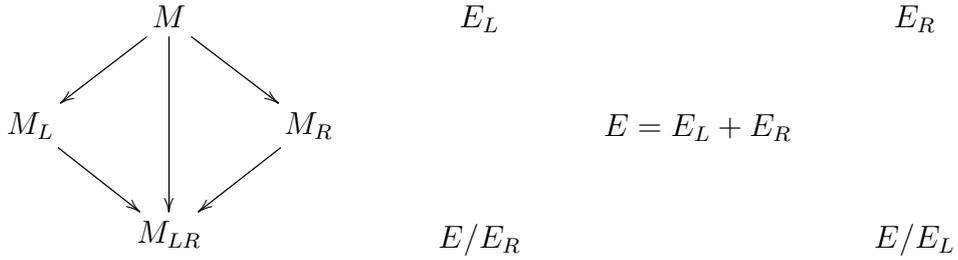
\xymatrix{
 & M \ar[dl] \ar[dr] \ar[dd] & \\
 M_L \ar[dr] & & M_R \ar[dl]\\
 & M_{LR}
 }
 & \xymatrix{
 E_L  & & E_R \\
 & E=E_L+E_R &\\
 E/E_R & & E/E_L
 }
\end{align*}
are the leaf spaces of the foliations on the right. Then $L_M$ and $R_M$ push forward to concurring Dirac structures at every vertex, and those on the bottom vertex are commuting Poisson structures.
\end{theorem*}

We turn next to illustrating our results in a few settings: the {\bf Hamiltonian actions} and {\bf moment maps} of \cite{BF}, which were already the model example in \cite{MW} and \cite{MR}; the {\bf Dirac-Nijenhuis} manifolds of \cite{BDN}, whose Poisson--Nijenhuis ancestors had been treated, in the context of reduction, in \cite{Vaisman96}; and the case of {\bf complex Dirac structures}, building on \cite{Wein07,Aguero_biv}, which suggests a novel way to generalize Magri's recipe. We conclude the paper with with a short discussion contextualizing our contribution within the existing literature.

\subsection*{Acknowledgements}

We would like to thank David Mart\'inez Torres, Rui Loja Fernandes, Olivier Brahic, Joel Villatoro, Ping Xu, Marco Zambon and Ioan M\u{a}rcu\cb{t} for useful conversations and feedback while this paper was being written.

A. Arsie and I. Mencattini are happy to acknowledge the kind support provided by FAPESP, grant number $2025/07987-2.$
\section{Preliminaries}

We collect in this section the main definitions we employ in the text. For an introduction to Dirac geometry, see e.g. \cite{Courant}, \cite{Bursztyn} or \cite{Meinrenken_perspective}.

\subsection{Dirac structures}

Call a subset $L \subset \mathbb{T}M$ a {\bf Lagrangian family} if $L$ meets each fibre $\mathbb{T}_xM$ in a Lagrangian vector subspace. Collect all such Lagrangian families $L$ into a set $\mathrm{Lag}(\mathbb{T}M)$. A smooth map $\phi : M \to P$ gives rise to a relation between $\mathbb{T}M$ and $\mathbb{T}P$, where
\begin{align*}
 && u_M + \xi_M \sim_{\phi} u_P + \xi_P && \Longleftrightarrow && \begin{cases}
                                                                   \phi_*(u_M)=u_P, \\
                                                                   \xi_M = \phi^*(\xi_P).
                                                                  \end{cases}
\end{align*}
Given Lagrangian families $L_M \in \mathrm{Lag}(\mathbb{T}M)$ and $L_P \in \mathrm{Lag}(\mathbb{T}P)$, we say that
\begin{align*}
 \phi : (M,L_M) \to (P,L_P)
\end{align*}
\begin{description}
 \item [is backward] if, for every $x \in M$ and $a_M \in L_{M,x}$, there exists $a_P \in L_{P,\phi(x)}$ such that $a_M \sim_{\phi} a_P$;
 \item [is forward] if, for every $x \in M$ and $a_P \in L_{P,\phi(x)}$, there exists $a_M \in L_{M,x}$ such that $a_M \sim_{\phi} a_P$,
\end{description}
in which cases we write, respectively,
\begin{align*}
 L_M = \phi^!(L_P) & = \{a_M \in \mathbb{T}M \ | \ \text{exists} \ a_P \in L_P, \ a_M \sim_{\phi} a_P\},\\
 \{a_P \in \mathbb{T}P \ | \ \text{exists} \ a_M \in L_M, \ a_M \sim_{\phi} a_P\} & = \phi_!(L_M) = L_P.
\end{align*}

If $L_M$ is a Lagrangian family on $M$, we write
\begin{align*}
 \Gamma(L_M) = \{ a \in \Gamma(\mathbb{T}M) \ | \ a(M) \subset L_M\}
\end{align*}
for the set of smooth sections of $\mathbb{T}M$ which take values in $L_M$. When $L_M$ is {\bf smooth}, i.e. a vector subbundle of $\mathbb{T}M$, $\Gamma(L_M)$ coincides with smooth sections of the vector bundle $L_M$. A smooth Lagrangian family is a {\bf Dirac structure} if its space of sections $\Gamma(L_M)$ is closed under the Dorfman bracket \eqref{eq : Dorfman}. Because $L_M$ is maximally isotropic, a smooth Lagrangian family is Dirac exactly when its {\bf Courant tensor}
\begin{align*}
 & \Upsilon_{L_M} \in \Gamma(\wedge^3L_M^*),
 & \Upsilon_{L_M}(a,b,c):=\langle [a,b],c\rangle
\end{align*}
vanishes identically.

\subsection{Dirac products}

We have well-defined {\bf tangent}- and {\bf cotangent products}
\begin{align*}
 \star,\circledast : \mathrm{Lag}(\mathbb{T}M) \to \mathrm{Lag}(\mathbb{T}M),
\end{align*}
that is, commutative and associative operations, explicitly given by
\begin{align*}
 L \star R & = \{a+\mathrm{pr}_{T^*}(b) \ | \ (a,b) \in L \times R, \ \mathrm{pr}_T(a-b)=0\}\\
 L \circledast R & = \{a+\mathrm{pr}_{T}(b) \ | \ (a,b) \in L \times R, \ \mathrm{pr}_{T^*}(a-b)=0\}.
\end{align*}
As explained in \cite{FMT_Concurrence}, the tangent product $L \star R$ of two Dirac structures is Dirac, provided that it be smooth. When that happens, leaves of $L$ and leaves of $R$ meet cleanly, and the leaves of $L \star R$ are the connected components of the intersections of leaves of $L$ and leaves of $R$. We note that $L \star R=\mathcal{R}_{\omega}(L)$ if $R$ is the graph of a two-form $\omega$, and that tangent products are compatible with backward maps, in the sense that, for a smooth map $\phi : N \to M$, we have an equality of Lagrangian families
\begin{align*}
 \phi^!(L \star R) = \phi^!(L) \star \phi^!(R).
\end{align*}
The cotangent product $L \circledast R$ of two Dirac structures need not be Dirac even if $L \circledast R$ is smooth. For example, when $L = \mathrm{Gr}(\pi_L)$ and $R = \mathrm{Gr}(\pi_R)$ correspond to Poisson structures $\pi_L$ and $\pi_R$, $L \circledast R  = \mathrm{Gr}(\pi_L+\pi_R)$ is Dirac exactly when $\pi_L$ and $\pi_R$ commute. More generally, two Dirac structures $L$ and $R$ are said to {\bf concur} if $L \circledast R$ is again a Dirac structure. In \cite{FMT_Concurrence} it is argued that concurrence is the natural notion of ``compatibility'' between two Dirac structures, and we abide here by that point of view. A fact that will play a key role in the sequel is that the cotangent product is compatible with forward maps, in that
\begin{align*}
 \phi_!(L \circledast R) = \phi_!(L) \circledast \phi_!(R)
\end{align*}
for any two Lagrangian families $L$ and $R$.


The following language (and notation) from \cite{CFZ} will be convenient in what follows. Let $I \subset \mathbb{T}M$ be an isotropic family. There is an induced map:
\begin{align}\label{eq : closest lagrangian}
 & \mathrm{Lag}(\mathbb{T}M) \to \mathrm{Lag}(\mathbb{T}M), & L[I]:=L\cap I^{\perp}+I.
\end{align}
Note that:
\begin{itemize}
 \item when $I \subset T^*M$, we have that
 \begin{align*}
  L[I] = L \star \mathrm{Gr}(I);
 \end{align*}
 \item when $I \subset TM$, we have that
 \begin{align*}
  L[I] = L \circledast \mathrm{Gr}(I).
 \end{align*}
\end{itemize}
Note also that, for a Dirac structure $L$, a sufficient condition for $L[I]$ to be smooth is that $L \cap I$ be smooth.

\subsection{Involutivity along an injective immersion}

We introduce the notion of involutivity along an injectively immersed submanifold. This is a localization along $X$ of the usual notion of involutivity along an embedded submanifold.

\begin{definition}\label{def : involutive along}\normalfont
 Let $i:X \to M$ be an injective immersion, and $D \subset \mathbb{T}M|_X$ a vector subbundle. We say that $D$ is {\bf involutive along $i$} if
\begin{align*}
 \Gamma(\mathbb{T}M,D|_V) = \{ s \in \Gamma(\mathbb{T}M) \ | \ s|_V \in \Gamma(D|_V)\}
\end{align*}
 is a subalgebra of $\Gamma(\mathbb{T}M)$ for every connected, open set $V \subset X$ on which $i$ restricts to an embedding $i:V \to M$  --- that is, $[a,b] \in \Gamma(\mathbb{T}M,D|_V)$ whenever $a,b \in \Gamma(\mathbb{T}M,D|_V)$.
\end{definition}

\begin{example}\normalfont
 For any submanifold $X \subset M$, $N^*X^{\perp} = TX \oplus T^*M|_X$ is involutive. Equivalently, if the anchors of $a_1,a_2 \in \Gamma(\mathbb{T}M)$ are tangent to $X$, so is their Dorfman bracket $[a_1,a_2]$.
\end{example}

\begin{example}\label{ex : i shriek LX}\normalfont
 Let $i:X \to M$ be an injective immersion, and let $L_X$ be a Dirac structure on $X$. Then
 \begin{align*}
  & i_!(L_X) \subset \mathbb{T}M|_X, & i_!(L_X)=\{u+\xi \ | \ u+i^*(\xi) \in L_X\}
 \end{align*}
is a Dirac structure along $i$.
\end{example}

\begin{lemma}\label{lem : inherits lie alg str}
 Let $i:X \to M$ be an injective immersion. An isotropic subbundle $D \subset \mathbb{T}M|_X$ for which $\mathrm{pr}_T(D) \subset TX$ inherits a Lie algebroid structure if it is involutive along $i$.
\end{lemma}
\begin{proof}
\noindent \emph{Case 1: $i$ is an embedding.} By the Leibniz property, $[a_1,a_2]|_X=0$ if $a_2|_X = 0$. Indeed, it suffices to argue on a neighborhood in $M$ of each point of $X$. There we have a frame $\zeta_1,...,\zeta_{2m}$ of $\mathbb{T}M$, so
\begin{align*}
 [a_1,a_2] = \sum_{i=1}^{2m}[a_1,f_i\zeta_i] = \sum_{i=1}^{2m}f_i[a_1,\zeta_i] + \sum_{i=1}^{2m}\langle \mathrm{d} f_i,\mathrm{pr}_T(a_1)\rangle \zeta_i.
\end{align*}
Because $\mathrm{pr}_T(a_1)$ is tangent to $X$, it follows that, if $f_i|_X = 0$ for all $i=1,...,2m$, then $[a_1,a_2]|_X =0$. This implies that
\begin{align*}
 [a_1|_X,a_2|_X]:=[a_1,a_2]|_X
\end{align*}
gives a well-defined Lie algebra structure on $\Gamma(D)$, and the induced bracket of $D$ can be computed using \emph{arbitrary} extensions. 

\noindent \emph{Case 2: $i$ is an injective immersion.} Because the condition of $D$ being a Lie subalgebroid is local in $X$, it suffices to show that $D|_V \subset \mathbb{T}M|_V$ is a Lie subalgebroid for an open neighborhood $V$ of every point $x \in X$. Choose a connected $V \subset X$ for which $i:V \to M$ is an embedding, and observe that the condition of involutivity along $i$ on $D$ implies that we are in the setting of Case 1, and therefore $D|_V$ is a Lie subalgebroid.
\end{proof}

\begin{definition}\label{def : Dirac along}\normalfont
 A Lagrangian subbundle $D \subset \mathbb{T}M|_X$ is called a {\bf Dirac structure along $i$} if $\mathrm{pr}_T(D) \subset TX$ and $D$ is involutive along $i$.
\end{definition}

\begin{remark}\normalfont
When $X$ is an \emph{embedded} submanifold, the condition of involutivity along $i$ in Definition \ref{def : involutive along} reduces to ``involutivity along $X$'', that is, $\Gamma(\mathbb{T}M,D)$ being a subalgebra. When $X$ is not embedded, involutivity along $i$ is a finer notion, as the example below illustrates.
\end{remark}

\begin{example}\normalfont
 Let $M = \mathbb{R}^5$ and let $X = \mathbb{R}^4 \times \mathbb{Q}$. Define nowhere-vanishing sections $w_1,w_2 \in \Gamma(TM|_X)$ by
 \begin{align*}
  & w_1\left(x_1,x_2,x_3,x_4,\tfrac{p}{q}\right) = \tfrac{\partial}{\partial x_3}+x_2\tfrac{\partial}{\partial x_4},
  & w_2\left(x_1,x_2,x_3,x_4,\tfrac{p}{q}\right) = \tfrac{\partial}{\partial x_1}+q\tfrac{\partial}{\partial x_2},
 \end{align*}
where $\tfrac{p}{q} \in \mathbb{Q}$ is in reduced form\footnote{That is, $p$ and $q > 0$ have no common prime divisor if $p \neq 0$, and $q=1$ if $p=0$.}, and subbundles $F,E \subset TM|_X$ by
\begin{align*}
 & F = \langle w_1 \rangle, & E = \langle w_1,w_2 \rangle.
\end{align*}
Observe that a vector field $v \in \Gamma(TM)$, $v = \sum_{i=1}^5 v_i \tfrac{\partial}{\partial x_i}$ lies in $E$ at $\left(x_1,x_2,x_3,x_4,\tfrac{p}{q}\right)$ iff
\begin{align*}
 v_2\left(x_1,x_2,x_3,x_4,\tfrac{p}{q}\right) & = qv_1\left(x_1,x_2,x_3,x_4,\tfrac{p}{q}\right),\\
 v_4\left(x_1,x_2,x_3,x_4,\tfrac{p}{q}\right) & = x_2v_3\left(x_1,x_2,x_3,x_4,\tfrac{p}{q}\right).
\end{align*}
Fix a reduced fraction $\tfrac{p}{q} \in \mathbb{Q}$ and observe that the sequence
\begin{align*}
 & \tfrac{1+npq}{nq^2}, & n \in \mathbb{N}
\end{align*}
is made of reduced fractions only, and
\begin{align*}
 \tfrac{p}{q} = \lim_{n \to \infty}\tfrac{1+npq}{nq^2}.
\end{align*}
Therefore
\begin{align*}
 v_2\left(x_1,x_2,x_3,x_4,\tfrac{p}{q}\right) & = \lim_{n \to \infty}v_2\left(x_1,x_2,x_3,x_4,\tfrac{1+npq}{nq^2}\right)\\
 & = \lim_{n \to \infty}nq^2v_1\left(x_1,x_2,x_3,x_4,\tfrac{1+npq}{nq^2}\right)
\end{align*}
implies that
\begin{align*}
 v_1\left(x_1,x_2,x_3,x_4,\tfrac{p}{q}\right) = \lim_{n \to \infty}v_1\left(x_1,x_2,x_3,x_4,\tfrac{1+npq}{nq^2}\right) = 0,
\end{align*}
and therefore
\begin{align*}
 v_2\left(x_1,x_2,x_3,x_4,\tfrac{p}{q}\right) = qv_1\left(x_1,x_2,x_3,x_4,\tfrac{p}{q}\right) = 0.
\end{align*}
This implies that $\Gamma(TM,E) = \Gamma(TM,F)$ is a Lie subalgebra. However, for each connected component $Y \subset X$, we have that $\Gamma(TM,E|_Y) \neq \Gamma(TM,F|_Y)$, and $\Gamma(TM,E|_Y)$ is not a Lie subalgebra, since
\begin{align*}
 \left[ \tfrac{\partial}{\partial x_1}+q\tfrac{\partial}{\partial x_2},\tfrac{\partial}{\partial x_3}+x_2\tfrac{\partial}{\partial x_4} \right] = q\tfrac{\partial}{\partial x_4} \not\in \Gamma(TM,E|_Y).
\end{align*}
Said otherwise, $\mathrm{Gr}(E) \subset \mathbb{T}M$ is not a Dirac structure along $i$, even though $\Gamma(TM,E)$ is a Lie subalgebra.
\end{example}

Observe that the operations of Dirac products and of closest Lagrangians \eqref{eq : closest lagrangian} have obvious extensions to the ``along $X$'' setting. For example,
\begin{align*}
 && L[I] = L \star \mathrm{Gr}(TX), && \text{if} && I=N^*X,\\
 && L[I] = L \circledast \mathrm{Gr}(I), && \text{if} && I\subset TM|_X.
\end{align*}

\section{Induced and projectable Dirac structures}

\vspace{0.2cm}
\begin{minipage}{0.9\textwidth}
\begin{mdframed}[backgroundcolor=blue!5]
In this section, we discuss the operations of inducing a Dirac structure by an injective immersion $i:X \to (M,L_M)$, and pushing forward a Dirac structure through a surjective submersion $p:(X,L_X) \to Y$. These operations will find a common generalization (``Dirac reduction'') in the next section.
\end{mdframed} 
\end{minipage}
\vspace{0.2cm}

One of the crucial features of the Dirac formalism is that there is a canonical candidate for a Dirac structure induced on a submanifold by a Dirac structure on the ambient manifold. Here we characterize when that candidate is actually Dirac, emphasizing necessary and sufficient conditions read off from the ambient Dirac structure along the submanifold.

\begin{lemma}\label{lem : induced iff}
 For an injective immersion $i:X \to M$ into a Dirac manifold $(M,L_M)$, the following are equivalent:
 \begin{enumerate}[i)]
  \item $L_M[N^*X]$ is a vector subbundle of $\mathbb{T}M|_X$;
  \item $L_M[N^*X]$ is a Dirac structure along $i$;
  \item $X$ has an induced Dirac structure $i^!(L_M)$.
 \end{enumerate}
\end{lemma}
\begin{proof}
As in the proof of Lemma \ref{lem : inherits lie alg str}, all three conditions in the statement hold exactly when they hold for every connected open set $V\subset X$ that $i$ embeds into $M$, so we may assume in the course of the proof below that $X$ is embedded.\\

\noindent \emph{i) implies ii)} $L_M[N^*X] = L_M \cap N^*X^{\perp}+N^*X$ is involutive, since
\begin{align*}
  && a,b \in \Gamma(\mathbb{T}M), && a|_X \in \Gamma(N^*X^{\perp}), && b|_X \in \Gamma(N^*X) && \Rightarrow && [a,b]|_X \in \Gamma(N^*X)
 \end{align*}
and $N^*X^{\perp}$ is involutive. \\

\noindent \emph{ii) implies iii)} If $L_M[N^*X]$ is Dirac, then
\begin{align*}
 i^!(L_M[N^*X]) = i^!(L_M)
\end{align*}
is smooth, since the canonical map
\begin{align*}
 & L_M[N^*X] \to i^!(L_M), & u+\xi \mapsto u+i^*(\xi)
\end{align*}
is surjective. As a smooth pullback of a Dirac structure it follows that $i^!(L_M)$ is Dirac. \\

\noindent \emph{iii) implies i)} If $i^!(L_M)$ is Dirac, then
 \begin{align*}
  i_!(i^!(L_M)) & = \{i_*(v) + \eta + \zeta \in \mathbb{T}M \ | \ v + \eta \in L_M, \ \zeta \in N^*X \}\\
  & = L_M \cap N^*X^{\perp}+N^*X\\
  & = L_M[N^*X]
 \end{align*}
is a smooth subbundle (as is the forward image of any Lagrangian subbundle of $\mathbb{T}X$). 
\end{proof}

\begin{remark}\normalfont
The Dirac structure $i^!(L)$ induced on an embedded submanifold $i:X \to M$ by a Dirac structure $L_M$ on $M$ may bear little resemblance to the ambient Dirac structure. The key issue is that a leaf of the induced Dirac structure need not lie inside a leaf of the ambient Dirac structure. These pathologies are discussed in \cite[Section 2]{BFM}, from which we extract the following illustrative example:
\end{remark}

\begin{example}\normalfont
 Let $M=\mathbb{R}^4$ be equipped with the Dirac structure $L_M$ which is the graph of the Poisson structure
 \begin{align*}
  \pi_M = \tfrac{\partial}{\partial x_1} \wedge \tfrac{\partial}{\partial x_2}+x_3\tfrac{\partial}{\partial x_3} \wedge \tfrac{\partial}{\partial x_4}.
 \end{align*}
 Let $f \in C^{\infty}(\mathbb{R}^2)$ be any smooth function, and consider the embedding
 \begin{align*}
  & i:\mathbb{R}^2 \to M, & i(x_1,x_2) = (x_1,x_2,f(x_1,x_2)^2,f(x_1,x_2)^2).
 \end{align*}
Because
\begin{align*}
 i^!(L_M) = \mathrm{Gr}\left(\tfrac{\partial}{\partial x_1} \wedge \tfrac{\partial}{\partial x_2} \right),
\end{align*}
the submanifold $X:=i(\mathbb{R}^2)$ has an induced symplectic structure. However, if $f^{-1}(0)$ is nonempty, $X$ hits more than one leaf of $L_M$.
\end{example}

A wide class of induced Dirac structures for which no such pathologies arise is that of \emph{split} induced structures --- the ones for which sections of the induced structure come from sections of the ambient structure:

\begin{definition}\label{def : splits}
 An injective immersion $i:X \to M$ {\bf splits} a Dirac structure $L_M \subset \mathbb{T}M$ if $i^!(L_M)$ is Dirac on $X$, and every section $a_X \in \Gamma(i^!(L))$ is $i$-related to a section of $L_M$.
\end{definition}

\begin{lemma}[Split condition]\label{lem : split}
For an embedded submanifold $X \subset M$ and a Lagrangian subbundle $L_M \subset \mathbb{T}M$, the following are equivalent:
\begin{enumerate}[i)]
 \item a vector subbundle $D \subset L_M$ exists, such that $L_M[N^*X] = D \oplus N^*X$;
 \item $i^!(L_M)$ is smooth, and every one of its sections is related to a section of $L_M$.
\end{enumerate}
\end{lemma}
\begin{proof}
 If $L_M[N^*X] = D \oplus N^*X$ for some subbundle $D \subset L_M|_X$, then $i^!(L_M)$ is smooth by Lemma \ref{lem : induced iff}. Moreover, every section $a_X$ of $i^!(L_M)$ can be lifted to a section $a_D$ of $D$. Because $D \subset L_M|_X$, the section $a_D$ of $D$ has a local extension to a section $a_M$ of $L_M$ in a neighborhood of $X$. Therefore $a_M \sim a_X$, and so i) implies ii). \\
 
 Conversely, suppose sections of $i^!(L_M)$ are related to sections of $L_M$. We construct a linear splitting $\sigma : i^!(L) \to N^*X$ to the canonical map $i^! : L[N^*X] \to i^!(L)$, with the property that
 \begin{align}\label{eq : convex condition}
  \mathrm{im}(\sigma) \subset N^*X^{\perp}.
 \end{align}
Because condition \eqref{eq : convex condition} is convex, a linear map $\sigma$ as above can be constructed from linear maps
\begin{align*}
 & \sigma_i : i^!(L_M)|_{U_i} \to N^*U_i, & \sigma_i \ \ \text{satisfies} \ \eqref{eq : convex condition}
\end{align*}
defined on the open sets of an open cover $\mathscr{U}=(U_i)$ of $X$, and a partition of unity $(\rho_i)$ subordinated to $\mathscr{U}$:
\begin{align*}
 \sigma = \sum_i \rho_i\sigma_i.
\end{align*}
It suffices then to argue the case where $i^!(L_M)$ has a global frame $a_1,...,a_n \in \Gamma(i^!(L_M))$, in which case $\sigma$ can be uniquely defined by requiring that $\sigma(a_i)=\widetilde{a}_i$ for any $\widetilde{a}_1,...,\widetilde{a}_n \in \Gamma(L_M)$ with $\widetilde{a}_i \sim a_i$.
\end{proof}

We recall next the Dirac version of Libermann's theorem, which provides necessary and sufficient conditions for a Dirac structure to be pushed forward to a Dirac structure by a surjective submersion with connected fibres. The question is: given a surjective submersion with connected fibres $p:X \to Y$, when does a Dirac structure $L_X$ on $X$ push forward under $p$ to a Dirac structure on $Y$ ? That is: when is there a Dirac structure $L_Y$ on $Y$, for which
\begin{align*}
 p : (X,L_X) \to (Y,L_Y)
\end{align*}
is a forward Dirac map?

\begin{lemma}[Libermann's theorem]\label{lem : libermann}
Let $p:X \to Y$ be a surjective submersion with connected fibres, and $L_X$ a Dirac structure on $X$. The following conditions are equivalent:
\begin{enumerate}[i)]
 \item $L_X$ pushes forward under $p$ to a Dirac structure $L_Y$ on $Y$;
 \item $L_X[F] = L_X \circledast \mathrm{Gr}(F)$ is Dirac, where $F \subset TX$ denotes the $p$-vertical foliation,
\end{enumerate}
in which case
\begin{align*}
 p^!(L_Y) = L_X[F].
\end{align*}
\end{lemma}
\begin{proof}
 If $p_!(L_X)=L_Y$, then $L_X[F]=p^!p_!(L_X) = p^!(L_Y)$ is Dirac. Conversely, because $F \subset L_X[F]$, if $L_X[F]$ is Dirac, then
 \begin{align*}
 p_!(L_{X,x}) = p_!(L_X[F]_{x}) = p_!(L_X[F]_{x'}) = p_!(L_{X,x'})
 \end{align*}
 for any $x,x' \in X$ in the same $F$-leaf (see \cite[Proposition 1]{FM_Dirac} for details). Thus $p_!(L_X)=p_!(L_X[F])$ defines a Lagrangian subbundle $L_Y$ on $Y$, which is Dirac because any two sections $a_Y,b_Y \in \Gamma(L_Y)$ are $p$-related to sections $a_X,b_X \in \Gamma(L_X[F])$, and therefore
 \begin{align*}
  && [a_X,b_X] \in \Gamma(L_X[F]), && [a_X,b_X] \sim [a_Y,b_Y] && \Longrightarrow && [a_Y,b_Y] \in \Gamma(L_Y).
 \end{align*}
 Hence ii) implies i).
\end{proof}

\section{Dirac reduction}\label{sec : Dirac reduction}

\begin{minipage}{0.9\textwidth}
\begin{mdframed}[backgroundcolor=blue!5]
We characterize the \emph{minimal} reduction scheme, which we call ``Dirac reduction'', and observe it to be \emph{incompatible} with concurrence.
\end{mdframed} 
\end{minipage}
\vspace{0.2cm}

The two basic operations of pulling back a Dirac structure by an injective immersion, and of pushing forward a Dirac structure under a surjective submersion have a common generalization in the notion of \emph{Dirac reduction}, which we discuss next.

\vspace{0.2cm}
\begin{minipage}{0.9\textwidth}
\begin{mdframed}[backgroundcolor=gray!10]
We fix a diagram
\begin{align}\tag{$\mathfrak{T}$}\label{triangle}
 M \stackrel{i}{\lmap} X \stackrel{p}{\rmap} Y
\end{align}
where:
\begin{itemize}
 \item $i:X \to M$ is an injective immersion;
 \item $p:X \to Y$ is a surjective submersion with connected fibres.
\end{itemize}
Denote by
\begin{align*}
 & F \subset TX, & F:=\ker(p_*:TX \to TY)
\end{align*}
the vertical bundle to $p$, and by $I$ the isotropic subbundle
\begin{align*}
 & I \subset \mathbb{T}M|_X, & I:=F \oplus N^*X.
\end{align*}
We say that $L_M$ has a {\bf Dirac reduction} to $Y$ if there is a Dirac structure $L_Y$ on $Y$ such that
\begin{align*}
 p_!i^!(L_M) = L_Y.
\end{align*}
\end{mdframed} 
\end{minipage}

\begin{theorem}[Dirac reduction]\label{thm : Dirac reduction}
 Given a diagram \eqref{triangle} as above, the following conditions on a Dirac structure $L_M$ on $M$ are equivalent:
 \begin{enumerate}[i)]
  \item $L_M$ has a Dirac reduction to $Y$;
  \item $i^!(L_M)[F]$ is Dirac on $X$;
  \item $L_M[I]$ is Dirac along $i$.
 \end{enumerate}
\end{theorem}
\begin{proof}
If a Dirac structure $L_Y$ on $Y$ exists, such that
\begin{align*}
 L_Y = p_!i^!(L_M),
\end{align*}
then
\begin{align*}
 p^!(L_Y) = p^!p_!i^!(L_M) = i^!(L_M)[F]
\end{align*}
by Lemma \ref{lem : libermann}. Therefore i) implies ii). Now, because
\begin{align*}
i^!(L_M)[F] = i^!(L_M[F]) = i^!(L_M[F][N^*X])= i^!(L_M[I])
\end{align*}
we see that as in Example \ref{ex : i shriek LX} that
\begin{align*}
i_!\left(i^!(L_M)[F]\right) = i_!i^!(L_M[I]) = L_M[I]
\end{align*}
is Dirac along $i$. Thus ii) implies iii). Finally, if $L_M[I] = L_M[F][N^*X]$ is Dirac along $i$, then by Lemma \ref{lem : induced iff},
\begin{align*}
i^!(L_M[F][N^*X]) = i^!(L_M[F]) = i^!(L_M)[F]
\end{align*}
is Dirac on $X$, so again by Lemma \ref{lem : libermann}, a Dirac structure $L_Y$ on $Y$ exists, such that
\begin{align*}
p^!(L_Y) = i^!(L_M)[F] = p^!p_!i^!(L_M),
\end{align*}
and therefore
\begin{align*}
L_Y = p_!i^!(L_M),
\end{align*}
so iii) implies i).
\end{proof}

\begin{example}\normalfont\label{ex : X need not be Dirac}
 We highlight that Dirac reduction under \eqref{triangle} does not assume (or imply) that $X$ itself inherits a Dirac structure from $L_M$. For example, consider the triangle
 \begin{align*}
 \mathbb{R}^2 \stackrel{i}{\lmap} X \stackrel{p}{\rmap} Y=\{ \ast \}
\end{align*}
where $X = \{x \in M \ | \ x_2=0\}$. Then any Dirac structure $L_M$ on $M$ Dirac reduces to the only Dirac structure on the one-point space $Y$. However, $X$ need not inherit a Dirac structure: for example, if
 \begin{align*}
  & L_M=\mathrm{Gr}(\pi_M), & \pi_M = x_1\tfrac{\partial}{\partial x_1} \wedge \tfrac{\partial}{\partial x_2},
 \end{align*}
 then
 \begin{align*}
  i^!(L_M)_x = \begin{cases}
              T_xX, & \text{if} \ x \neq 0;\\
              T^*_xX & \text{if} \ x=0
             \end{cases}
 \end{align*}
and therefore $i^!(L_M)$ is not smooth.
\end{example}

In this paper, we espouse the point of view that ``concurrence'' is the canonical compatibility relation between Dirac structures. A natural question that arises at this point is:
\begin{quotation}
 \emph{Does the Dirac reduction scheme of Theorem \ref{thm : Dirac reduction} respect concurrence? Namely, if two Dirac structures on $M$ concur weakly, and have Dirac reductions to $Y$, do the reduced Dirac structures on $Y$ concur weakly? }
\end{quotation}

That turns out \underline{not} to be the case, as the following example illustrates:

\begin{example}[Dirac reduction does not respect concurrence]\label{ex : reduction does not concur}\normalfont
 Consider on $M=\mathbb{R}^4$ the commuting Poisson structures
 \begin{align*}
  & \pi_L = \tfrac{\partial}{\partial x_1} \wedge \tfrac{\partial}{\partial x_4}, & \pi_R = \left(\tfrac{\partial}{\partial x_2}+x_1\tfrac{\partial}{\partial x_3}\right) \wedge \tfrac{\partial}{\partial x_4}.
 \end{align*}
Let $X \subset M$ be the submanifold given by $x_4=0$. Then
 \begin{align*}
  & i^!\mathrm{Gr}(\pi_L) = \mathrm{Gr}(\tfrac{\partial}{\partial x_1}),  & i^!\mathrm{Gr}(\pi_R) =\mathrm{Gr}(\tfrac{\partial}{\partial x_2}+x_1\tfrac{\partial}{\partial x_3})
 \end{align*}
 do not concur:
 \begin{align*}
  \langle [\tfrac{\partial}{\partial x_1},\tfrac{\partial}{\partial x_2}+x_1\tfrac{\partial}{\partial x_3}],\mathrm{d} x_3-x_1\mathrm{d} x_2\rangle = 1.
 \end{align*}
\end{example}

In the sequel, we will address this lack of compatibility of Dirac reduction with concurrence by devising a more selective reduction scheme under which concurrence is compatible. As in the classical reduction scheme proposed by Marsden and Ra\cb{t}iu in the context of Poisson structures, the generalization we propose to Dirac structures requires that the Dirac reduction be \emph{witnessed} by a certain subbundle $E \subset TM|_X$.

\section{Marsden-Ra\cb{t}iu reduction}

\begin{minipage}{0.9\textwidth}
\begin{mdframed}[backgroundcolor=blue!5]
We recall the Marsden-Ra\cb{t}iu reduction scheme for Poisson structures, and highlight its concurring nature that we will later seek to generalize to Dirac structures.
\end{mdframed} 
\end{minipage}
\vspace{0.2cm}

In the context of Poisson geometry, one reduction scheme was proposed by Jerrold Marsden and Tudor Ra\cb{t}iu in \cite{MR}. It depends on a choice of an {\bf adapted} subbundle to \eqref{triangle} --- that is,
\begin{align}\tag{$\mathfrak{E}$}\label{adapted subbundle}
 & E \subset TM|_X, & F = E \cap TX,
\end{align}
which we think of as an extra structure on \eqref{triangle}:

\begin{theorem}[Marsden-Ra\cb{t}iu reduction]\label{thm : MR}
 Let $(M,\pi_M)$ be a Poisson manifold, and denote by $\mathfrak{g}$ the Lie algebra $C^{\infty}(M)$, equipped with the Poisson bracket of $\pi_M$. If $i$ in \eqref{triangle} is an embedding, and an adapted subbundle \eqref{adapted subbundle} exists, such that {\normalfont
 \begin{enumerate}[MR1)]
  \item $\pi_M^{\sharp}(E^{\circ}) \subset TX + E$, and \label{MR1}
  \item the subset
  \begin{align*}
   & C^{\infty}(M)_E \subset \mathfrak{g}, & C^{\infty}(M)_E = \{ H \in C^{\infty}(M) \ | \ \mathrm{d}H(E) = 0 \}
  \end{align*}
is a Poisson subalgebra, \label{MR2}
 \end{enumerate}
}
then $\pi_M$ Dirac reduces to a Poisson structure $\pi_Y$ on $Y$, and we have that
 \begin{align*}
  p^!(\mathrm{Gr}(\pi_Y)) = i^!\left(\mathrm{Gr}(\pi_M)[E]\right).
 \end{align*}
\end{theorem}

\begin{remark}\label{rem : MR equivalent subalgebroid}\normalfont
It was observed in \cite[Lemma 2.2]{FZ} that, under condition \hyperref[MR2]{MR2)}, condition \hyperref[MR1]{MR1)} is equivalent to
\begin{enumerate}[MR1')]
  \item $\pi_M^{\sharp}(E^{\circ}) \subset TX$. \label{MR1'}
 \end{enumerate}
Otherwise said \cite[Remark 2.2]{FZ}, conditions \hyperref[MR1]{MR1)} and \hyperref[MR2]{MR2)} are equivalent to
\begin{align*}
 \mathcal{R}_{\pi_M}(E^{\circ})=\{\pi_M^{\sharp}(\xi)+\xi \ |\ \xi \in E^{\circ}\}
\end{align*}
being a Lie subalgebroid of $\mathrm{Gr}(\pi_M)$ over $X$. 
\end{remark}

\begin{proof}[Proof of Theorem \ref{thm : MR}]
 Because
 \begin{align*}
  \mathcal{R}_{\pi_M}(E^{\circ}) \subset \mathcal{R}_{\pi_M}(F^{\circ}) \cap (N^*X)^{\perp},
 \end{align*}
 by Remark \ref{rem : MR equivalent subalgebroid}, we have that
 \begin{align*}
  \mathrm{Gr}(\pi_M)[E][N^*X] = \mathrm{Gr}(\pi_M)[F][N^*X] = \mathrm{Gr}(\pi_M)[I] = F + \mathcal{R}_{\pi}(E^{\circ})+N^*X
 \end{align*}
 Because $F$ and $N^*X$ are involutive, and $\mathcal{R}_{\pi}(E^{\circ})$ is involutive by Remark \ref{rem : MR equivalent subalgebroid}, it follows from Theorem \ref{thm : Dirac reduction} that $\mathrm{Gr}(\pi_M)$ Dirac reduces to a Dirac structure $L_Y$ on $Y$. This is a Poisson structure because
 \begin{align*}
  i^!(\mathrm{Gr}(\pi_M)) \cap TX = F
 \end{align*}
and therefore $L_Y \cap TY = 0$. This implies that $L_Y$ is the graph of a Poisson structure $\pi_Y$ on $Y$.
\end{proof}

\begin{example}\normalfont
 Much like with Dirac reduction, the conditions for Marsden-Ra\cb{t}iu reduction do not ensure that $X$ itself is Dirac. As an illustration, in Example \ref{ex : X need not be Dirac}, the adapted subbundle $E=TX$ satisfies \hyperref[MR1]{MR1)} and \hyperref[MR2]{MR2)}.
\end{example}

In the taxonomy we propose here, ``Dirac reduction'' as formulated in Theorem \ref{thm : Dirac reduction} consists of a \emph{minimal} reduction scheme --- that is, the least stringent one which allows the transfer of a Dirac structure on the ambient manifold to the quotient of a submanifold. 

A key fact we wish to highlight in this note is that the Marsden-Ra\cb{t}iu reduction is a \emph{concurrent} reduction scheme --- that is, one which respects the compatibility relation of ``concurrence'', a fact which seems to have first been observed in \cite[Proposition 1.1.]{CMP}.

\section{Witnesses of concurrent reduction}\label{eq : Witnesses of concurrent reduction}

\begin{minipage}{0.9\textwidth}
\begin{mdframed}[backgroundcolor=blue!5]
In this section, we propose our generalization of the Marsden-Ra\cb{t}iu reduction scheme to Dirac structures, and discuss the properties it satisfies.
\end{mdframed} 
\end{minipage}
\vspace{0.2cm}

Given a triangle \eqref{triangle} as before,
\begin{align*}
 M \stackrel{i}{\lmap} X \stackrel{p}{\rmap} Y,
\end{align*}
the problem is to propose a reduction scheme that respects concurrence. More precisely, we wish to propose a scheme of reduction for Dirac structures, which meets the following requirements:

\vspace{0.2cm}
\begin{minipage}{0.9\textwidth}
\begin{mdframed}[backgroundcolor=red!10]
\begin{enumerate}[CR1)]
 \item It implies Dirac reduction under \eqref{triangle};
 \item It should boil down to the Marsden-Ratiu reduction when $X$ is embedded and $L_M$ corresponds to a Poisson structure;
 \item It should boil down to Libermann's theorem when $X=M$;
 \item It respects concurrence in the sense that the Dirac reductions under \eqref{triangle} of \emph{concurring} Dirac structures should also concur (perhaps weakly).
\end{enumerate} 
\end{mdframed} 
\end{minipage}
\vspace{0.2cm}

The compatibility condition the current reduction scheme demands is encoded in the following definition:
\vspace{0.2cm}

\begin{minipage}{0.9\textwidth}
 \begin{mdframed}[backgroundcolor=olive!10]
\begin{definition}\normalfont\label{def : witness}
 A subbundle \eqref{adapted subbundle} is a {\bf witness} for a Lagrangian family $L_M$ on $M$ if:
 \begin{enumerate}[\text{Wit}1)]
  \item $L_M[I]$ is a vector bundle, where $I:=F \oplus N^*X$; \label{Wit1}
  \item $L_M[E] \cap N^*X^{\perp}$ is involutive along $i$; \label{Wit2}
  \item $L_M \cap E^{\perp} \subset N^*X^{\perp}$. \label{Wit3}
 \end{enumerate}
\end{definition}
\end{mdframed}
\end{minipage}

\begin{proposition}[CR1 holds]\label{pro : CR1}
 A Dirac structure $L_M$ has a Dirac reduction under \eqref{triangle} if a witness \eqref{adapted subbundle} exists for it.
\end{proposition}
\begin{proof}
Let \eqref{adapted subbundle} be a witness for a Dirac structure $L_M$. By \hyperref[Wit1]{Wit1)}, $L_M[I]$ is a vector bundle; and by \hyperref[Wit3]{Wit3)},
\begin{align}\label{eq: wit3 implies}
 L_M[I] = L_M[E][N^*X] = L_M[E] \cap N^*X^{\perp}+ N^*X,
\end{align}
which is Dirac along $i$ by \hyperref[Wit2]{Wit2)}. Hence by Theorem \ref{thm : Dirac reduction}, $L_M$ has a Dirac reduction to a Dirac structure on $Y$.
\end{proof}

\begin{remark}\normalfont
 One could also wonder whether it is possible to do away with the extra choice of adapted subbundle \eqref{adapted subbundle}, and restrict our attention to \eqref{triangle} alone. After all, the reduced Dirac structure only depends on $F$, and not on $E$. It turns out that demanding that $E=F$ be canonical for a Dirac structure $L_M$ is excessively restrictive. For example,
\begin{align*}
 & L_M = \mathrm{Gr}(\pi_M), & \pi_M = \tfrac{\partial}{\partial x_1} \wedge \tfrac{\partial}{\partial x_2} + \tfrac{\partial}{\partial x_3} \wedge \tfrac{\partial}{\partial x_4}
\end{align*}
is a Dirac structure on $M=\mathbb{R}^4$. Consider $X \subset M$ given by $x_3=0=x_4$, and $p:X \to Y=\mathbb{R}$, $p(x_1,x_2,0,0) = x_1$. Then
\begin{align*}
 & L_M[I] = \left\langle \mathrm{d} x_1, \tfrac{\partial}{\partial x_2}, \mathrm{d} x_3, \mathrm{d} x_4\right\rangle,
\end{align*}
so $L_M$ reduces to $L_Y = T^*Y$. Note that
\begin{align*}
 & L_M \cap F^{\perp} = \left\langle \tfrac{\partial}{\partial x_2} + \mathrm{d} x_1, \tfrac{\partial}{\partial x_4} + \mathrm{d} x_3, \tfrac{\partial}{\partial x_3} - \mathrm{d} x_4\right\rangle
\end{align*}
does not satisfy \hyperref[MR1']{MR1')}. However, the vector subbundle
\begin{align*}
 & E \subset TM|_X, & E = \left\langle \tfrac{\partial}{\partial x_2}, \tfrac{\partial}{\partial x_3}, \tfrac{\partial}{\partial x_4}\right\rangle
\end{align*}
does satisfy \hyperref[MR1]{MR1)} and \hyperref[MR2]{MR2)}.
\end{remark}

\begin{proposition}[CR2 holds]\label{pro : CR2}
 When $i$ in \eqref{triangle} is an embedding, \eqref{adapted subbundle} satisfies the MR conditions for $\pi_M$ exactly when \eqref{adapted subbundle} is a witness for $\mathrm{Gr}(\pi_M)$, in which case the reduced Dirac structure is again Poisson.
\end{proposition}
\begin{proof}
Condition \hyperref[MR1']{MR1')} is equivalent to \hyperref[Wit3]{Wit3)}, and under it, \hyperref[Wit1]{Wit1)} is satisfied, because
\begin{align*}
 \mathrm{Gr}(\pi_M)[I] & \stackrel{\eqref{eq: wit3 implies}}{=} \mathrm{Gr}(\pi_M)[E][N^*X]\\
 & = \left(E \oplus \mathcal{R}_{\pi}(E^{\circ})\right)[N^*X]\\
 & = F \oplus \mathcal{R}_{\pi}(E^{\circ}) + N^*X
\end{align*}
and the latter is smooth because $F=E \cap TX$ is smooth, and therefore so is
\begin{align*}
 \left(F \oplus \mathcal{R}_{\pi}(E^{\circ})\right) \cap N^*X = E^{\circ} \cap N^*X.
\end{align*}
It remains to show that \hyperref[MR2]{MR2)} and \hyperref[Wit2]{Wit2)} are equivalent conditions. Under either hypothesis, we have that
\begin{align*}
 i^!(\mathrm{Gr}(\pi_M)[E]) = i^!(\mathrm{Gr}(\pi_M))[F] = p^!(L_Y)
\end{align*}
for a Dirac structure $L_Y$ on $Y$. Because
\begin{align*}
 i^!(\mathrm{Gr}(\pi_M)[E]) \cap TX = \{ v + \pi_M(\xi) \ | \ v \in F, \ \xi \in E^{\circ} \cap N^*X\} \subset F
\end{align*}
since \hyperref[MR1']{MR1')} implies that $\pi_M(E^{\circ} \cap N^*X) \subset F$, we deduce that the Dirac structure $L_Y$ meets $TY$ trivially, and is therefore the graph of a Poisson structure $\pi_Y$ on $Y$. 

Next observe that if $\xi,\eta \in \Omega^1(M)$ are such that $\xi|_X,\eta|_X \in \Gamma(E^{\circ})$, then \hyperref[Wit2]{Wit2)} implies that there exist $v \in \Gamma(F)$ and $\zeta \in \Gamma(E^{\circ})$ for which
\begin{align*}
 [\pi_M(\xi)+\xi,\pi_M(\eta)+\eta] = [\pi_M(\xi),\pi_M(\eta)]+[\xi,\eta]^{\pi_M} = v + \pi_M(\zeta)+\zeta,
\end{align*}
and this in turn implies that $\zeta = [\xi,\eta]^{\pi_M}$ and $v=0$. Therefore \hyperref[Wit2]{Wit2)} $\Rightarrow$ \hyperref[MR2]{MR2)}. To prove the converse implication, let
\begin{align*}
 \mathcal{S} = \{ \xi \in \Gamma(E^{\circ}) \ | \ i^*(\xi) \in p^*\Omega^1(Y) \}.
\end{align*}
Assume \hyperref[MR2]{MR2)} and choose $\xi \in \mathcal{S}$, $i^*(\xi) = p^*(\xi')$, and $v \in \Gamma(F)$. Then
\begin{align*}
 [v,\pi_M(\xi)+\xi]|_X = [v,\pi_M(\xi)]|_X \in \Gamma(F),
\end{align*}
since $p_*\pi_M(\xi)=\pi_Y(\xi')$. Because $\Gamma(E^{\circ})$ is spanned over $C^{\infty}(X)$ by $\mathcal{S}$, it follows from the Leibniz rule for the Dorfman bracket that
\begin{align*}
 [\Gamma(F),\Gamma(\mathcal{R}_{\pi_M}(E^{\circ}))] \subset \Gamma(F) \oplus \Gamma(\mathcal{R}_{\pi_M}(E^{\circ}))
\end{align*}
and so \hyperref[Wit2]{Wit2)} holds.
\end{proof}

\begin{proposition}[CR3 holds]
 If $X=M$, a Dirac structure $L_M$ has a Dirac reduction to $Y$ exactly when the vertical bundle to $p:M \to Y$ is a witness for $L_M$.
\end{proposition}
\begin{proof}
 When $X=M$, we have that:
 \begin{itemize}
  \item $E=F$;
  \item \hyperref[Wit3]{Wit3)} is always met, since $N^*X^{\perp}=\mathbb{T}M|_X$;
  \item \hyperref[Wit1]{Wit1)} is equivalent to $L_M[F]$ being smooth;
  \item \hyperref[Wit2]{Wit2)} is equivalent to $L_M[F]$ involutive.
 \end{itemize}
Therefore $F$ is a witness iff $L_M[F]$ is a Dirac structure, and by Lemma \ref{lem : libermann}, this is the case exactly when $p_!(L_M)$ is a Dirac structure on $Y$.
\end{proof}

\vspace{0.2cm}
\begin{minipage}{0.9\textwidth}
 \begin{mdframed}[backgroundcolor=olive!10]
Our proposed reduction scheme for a Dirac structure $L_M$ on $M$ to reduce under \eqref{triangle} to a Dirac structure on $Y$ is that the ambient Dirac structure have a witness \eqref{adapted subbundle}. 
\end{mdframed}
\end{minipage}
\begin{remark}\normalfont
It must be emphasized that the reduced Dirac structure on $Y$ is the very same as that given by Theorem \ref{thm : Dirac reduction}: because of \eqref{eq: wit3 implies}, the map
\begin{align*}
 & \mathcal{E} : \mathrm{Lag}(\mathbb{T}M) \to \mathrm{Lag}(\mathbb{T}M|_X), & \mathcal{E}(L):=L[E][N^*X].
\end{align*}
coincides with $L_M \mapsto L_M[I]$ if \eqref{adapted subbundle} is an adapted subbundle satisfying \hyperref[Wit3]{Wit3)}. The role of witnesses to the Dirac reduction is to ensure that concurrence is preserved, as we explain in the following section. It should also be pointed out that Dirac reduction does not necessarily admit a witness, as the next example shows.
\end{remark}

\begin{example}\label{reduction_2_forms}\normalfont
Let $\mathscr{F}$ be a foliation on a smooth manifold $Y$, and consider the triangle
\begin{align*}
    M:=T^*Y \stackrel{i}{\lmap} X:=N^*\mathscr{F} \stackrel{p}{\rmap} Y,
\end{align*}
where $i$ stands for the inclusion, and $p$ the canonical projection. Equip these manifolds with the Dirac structures
\begin{align*}
    && L_M:=\mathrm{Gr}(\omega_{\mathrm{can}}),
    && L_X:=\mathrm{Gr}(\omega_{\mathscr{F}}),
    && L_Y:=\mathrm{Gr}(\mathscr{F}),
\end{align*}
where $\omega_{\mathrm{can}}$ is the canonical symplectic form on $M$, and $\omega_{\mathscr{F}}:=i^*(\omega_{\mathrm{can}})$. Then
\begin{align*}
    i:(X,L_X) \to (M,L_M)
\end{align*}
is backward and
\begin{align*}
    p:(X,L_X) \to (Y,L_Y)
\end{align*}
is strong and forward --- namely, every $a_Y \in L_Y$ is $p$-related to a \emph{unique} $a_X \in L_X$ \cite[Example 6]{FM_Dirac}. In particular, $L_M$ has a Dirac reduction to $L_Y$ on $Y$. Note however that this Dirac reduction cannot be witnessed: indeed, by Proposition \ref{pro : CR2}, the reduced Dirac structure on $Y$ would be Poisson if this Dirac reduction were to admit a witness $E$ --- and that only happens if $\mathscr{F}$ is the foliation $L_Y=T^*Y$  by points of $Y$.
\end{example}

\begin{remark}\label{rmk : split on open dense}\normalfont
When $i$ in \eqref{triangle} is an embedding, and a Dirac structure $L_M$ has a witness \eqref{adapted subbundle}, the embedded submanifold $X$ has a Dirac structure induced from $(M,L_M[E])$, and the restriction of $i$ to an open, dense set $U \subset X$ splits it (Definition \ref{def : splits}). Indeed, if $U$ is the maximal open set on which $L_M[E] \cap N^*X$ is smooth, then a smooth subbundle $D \subset L_M[E] \cap N^*U^{\perp}$ exists, such that
 \begin{align*}
  L_M[E][N^*U] = D \oplus N^*U,
 \end{align*}
so we may invoke Lemma \ref{lem : split} over $U$. This allows us to relate certain sections of the reduced structure $L_Y$ to those on $L_M$, as we do in the sequel.
\end{remark}

\section{Compatibility with concurrence}

\begin{minipage}{0.9\textwidth}
\begin{mdframed}[backgroundcolor=blue!5]
In this section, we discuss the main feature of the extension of Marsden-Ra\cb{t}iu reduction to Dirac structures: namely, that it is compatible with (weak) concurrence. 
\end{mdframed} 
\end{minipage}

\begin{theorem}\label{thm : reduction of weakly concurring concur weakly}
 If \eqref{adapted subbundle} is a witness for two weakly concurring Dirac structures $L_M$ and $R_M$ on $M$, then the reduced Dirac structures $L_Y$ and $R_Y$ on $Y$ concur weakly as well.
\end{theorem}
\begin{proof}
Throughout the proof, given sections $a \in \Gamma(L)$ and $b \in \Gamma(R)$, we will say that $a$ and $b$ are $\circledast$-composable if $\mathrm{pr}_{T^*}(a)=\mathrm{pr}_{T^*}(b)$, in which case $a \circledast b$ denotes the section
\begin{align*}
    a \circledast b = a + \mathrm{pr}_T(b) = \mathrm{pr}_T(a) + b \in \Gamma(L \circledast R).
\end{align*}
    
    Let $E$ be a common witness for two Dirac structures $L_M$ and $R_M$ on $M$ which concur weakly. Let $L_Y$ and $R_Y$ be the respective reduced Dirac structures on $Y$ as in Theorem \ref{thm : Dirac reduction}. \\
    
\noindent \emph{Step 0.} We start with the observation that $E \subset TM|_X$ is a witness for two weakly concurring Dirac structures $L_M$ and $R_M$ on $M$ exactly when, for all connected, open set $V \subset X$ that $i$ embeds, we have that $E|_V \subset TM|_V$ is a witness for $L_M$ and $R_M$. It therefore suffices to argue the case where $i$ is an embedding.\\
    
    Assuming that $X$ is an embedded submanifold, our task is to show that, for any $a_Y, b_Y \in \Gamma(L_Y\circledast R_Y)$, the bracket $[a_Y, b_Y] \in \Gamma(L_Y\circledast R_Y)$. Having chosen such sections, we will construct, in a series of steps, sections
    \begin{align*}
        & a_X,b_X  \ \ \text{of} \ \ p^!(L_Y) \circledast p^!(R_Y), & a_M,b_M \ \ \text{of} \ \ L_M \circledast R_M,
    \end{align*}
    with
    \begin{align*}
        && a_X \sim_i a_M, && b_X \sim_i b_M, && a_X \sim_p a_Y, && b_X \sim_p b_Y.
    \end{align*}
    
\noindent \emph{Step 1.} There exists $U_Y$ open and dense in $Y$ such that
\begin{align*}
& a_Y = a_{Y,R}\circledast a_{Y,L}, & b_Y = b_{Y,R}\circledast b_{Y,L}, 
\end{align*}
where 
\begin{align*}
& a_{Y,L}, b_{Y,L} \in \Gamma(L_Y|_{U_Y}), & a_{Y,R}, b_{Y,R} \in \Gamma(R_Y|_{U_Y}).
\end{align*}
To see that, observe that the cotangent product $L_Y \circledast R_Y$ can be described as follows: there is a canonical vector bundle map
\begin{align*}
    & \delta : L_Y \oplus R_Y \to T^*Y, & \delta(a_L,a_R) = \mathrm{pr}_{T^*}(a_L-a_R),
\end{align*}
and $L_Y \circledast R_Y$ is the image of the restriction to $\ker(\delta) \subset L_Y \oplus  R_Y$ of the vector bundle map
\begin{align*}
    & \iota : L_Y \oplus R_Y \to \mathbb{T}Y, & \iota(a_L,a_R) = a_L + \mathrm{pr}_{T}(a_R).
\end{align*}
It then follows that, on the open, dense set $U_Y$ on which $\delta$ has locally constant rank, $L_Y \circledast R_Y$ is smooth (as the image under an injective map of a vector subbundle), and we may use
\begin{align*}
    \iota^{-1} : (L_Y \circledast R_Y)|_{U_Y} \to L_Y|_{U_Y} \oplus R_Y|_{U_Y}
\end{align*}
to construct a map of sections
\begin{align*}
    & \Gamma((L_Y \circledast R_Y)|_{U_Y}) \to \Gamma(L_Y|_{U_Y}) \times \Gamma(R_Y|_{U_Y}), & a_Y \mapsto (a_{Y, L},a_{Y, R}),
\end{align*}
with the property that $\mathrm{pr}_{T^*}(a_{Y,L})=\mathrm{pr}_{T^*}(a_{Y,R})$ and
\begin{align*}
    a_Y = a_{Y,L} \circledast a_{Y,R} := a_{Y,L} + \mathrm{pr}_{T}(a_{Y,R}).
\end{align*}

\noindent \emph{Step 2.} Choose an Ehresmann connection $h: \mathfrak{X}(Y)\to \mathfrak{X}(X).$ Then we can define
\begin{align*}
a_{X,L}^{\dagger}:=h(\mathrm{pr}_T (a_{Y,L}))+p^*(\mathrm{pr}_{T^*}(a_{Y,L}))\in \Gamma(p^{!}(L_Y)|_{U_X}),    
\end{align*}
where $U_X:=p^{-1}(U_Y)$. Observe that, by construction, $a_{X,L}^{\dagger}$ is $p$-related to $a_{Y,L}$. Using the same Ehresmann connection, construct analogously
\begin{align*}
    & b_{X,L}^{\dagger} \in \Gamma(p^{!}(L_Y)|_{U_X}),  & a_{X,R}^{\dagger}, b_{X,R}^{\dagger} \in \Gamma(p^{!}(R_Y)|_{U_X}),  
\end{align*}
respectively $p$-related to the sections $b_{Y,L}$, $a_{Y,R}$ and $b_{Y,R}$.\\

\noindent \emph{Step 3.} Since $E$ is a common witness, we have that 
\begin{align*}
    & i^{!}(L_M[E])=p^{!}(L_Y), & i^{!}(R_M[E])=p^{!}(R_Y).
\end{align*}
Therefore the four sections constructed in Step 2 can be regarded as sections
\begin{align*}
    & a_{X,L}^{\dagger},b_{X,L}^{\dagger} \in \Gamma(i^{!}(L_M[E])|_{U_X}),  & a_{X,R}^{\dagger}, b_{X,R}^{\dagger} \in \Gamma(i^{!}(R_M[E])|_{U_X}).    
\end{align*}
Fix a linear splitting $E=F \oplus C$, and observe that the canonical map $T^*M|_X \to T^*X$ restricts to a surjective map $C^{\circ} \to T^*X$. Let $\lambda : T^*X \to C^{\circ}$ be a linear splitting, and define a map
\begin{align*}
   &\varphi:  \Gamma(\mathbb{T}U_X) \to \Gamma(\mathbb{T}M|_{U_X}), & a \mapsto  \mathrm{pr}_T(a)+\lambda\left(\mathrm{pr}_{T^*}(a)\right),
\end{align*}
and denote with $a_{X,L}^{\ddagger}:=\varphi(a_{X,L}^{\dagger})$ and analogously for the other sections.
Then observe that the sections corresponding to $a_{X,L}^{\dagger},b_{X,L}^{\dagger}, a_{X,R}^{\dagger}, b_{X,R}^{\dagger}$ lie respectively in
\begin{align*}
    & a_{X,L}^{\ddagger},b_{X,L}^{\ddagger} \in \Gamma(L_M[E]|_{U_X}),  & a_{X,R}^{\ddagger}, b_{X,R}^{\ddagger} \in \Gamma(R_M[E]|_{U_X}),
\end{align*}
and are by construction $i$-related to the original sections: $a_{X,L}^{\dagger} \sim_i a_{X,L}^{\ddagger}$ and so on. Furthermore, note that there are well-defined sections
\begin{align*}
    a_{X,L}^{\ddagger} \circledast a_{X,R}^{\ddagger}, \ \ b_{X,L}^{\ddagger} \circledast b_{X,R}^{\ddagger} \ \in \ \Gamma(L_M[E] \circledast R_M[E]|_{U_X}).
\end{align*}

\noindent \emph{Step 4.} Since
\begin{align*}
& L_{M}[E]=L_M\cap E^{\perp}+E & 
R_M[E]=R_M\cap E^{\perp}+E,    
\end{align*}
there exists an open dense subset $U'_X$ in $U_X$ over which $L_M \cap E$ and $R_M \cap E$ are vector subbundles. Therefore one can find linear splittings
\begin{align*}
    & L_M[E]|_{U'_X}=(L_M\cap E^{\perp})|_{U'_X}\oplus C_L, 
    & R_M[E]|_{U'_X}=(R_M\cap E^{\perp})|_{U'_X}\oplus C_R
\end{align*}
where $C_L$ and $C_R$ are vector subbundles of $E|_{U'_X}$. Denote by
\begin{align*}
    & q_L: L_M[E]|_{U'_X}\to (L_M\cap E^{\perp})|_{U'_X}, 
    & q_R: R_M[E]|_{U'_X}\to (R_M\cap E^{\perp})|_{U'_X}
\end{align*}
the canonical projections, and define new sections
\begin{align*}
    & a_{X,L},b_{X,L} \in \Gamma(L_M \cap E^{\perp}|_{U'_X}),  & a_{X,R}, b_{X,R} \in \Gamma(R_M \cap E^{\perp}|_{U'_X})
\end{align*}
by applying $q_L$ to $a_{X,L}^{\ddagger}$ and $b_{X,L}^{\ddagger}$, and $q_R$ to $a_{X,R}^{\ddagger}$ and $b_{X,R}^{\ddagger}$. Observe that
\begin{align*}
    && a_{X,L} \sim_p a_{Y,L}, && b_{X,L} \sim_p b_{Y,L}, && a_{X,R} \sim_p a_{Y,R}, && b_{X,R} \sim_p b_{Y,R},
\end{align*}
and
\begin{align*}
    & a_{X,L} \ \text{and} \ a_{X,R} \ \text{are $\circledast$-composable}, & b_{X,L} \ \text{and} \ b_{X,R} \ \text{are $\circledast$-composable}.
\end{align*}

\noindent \emph{Step 5.} Because $X$ is embedded in $M$, there is a tubular neighborhood for $X$ in $M$, that is, an open set $V \subset M$  and a surjective submersion with connected fibres $q:V \to X$. Let $h' : \mathfrak{X}(X) \to \mathfrak{X}(V)$ be an Ehresmann connection for $q$, whose horizontal distribution, restricted to points of $X$ coincides with $TX\subset TV|_X$. Define as in Step 2 sections
\begin{align*}
    & a_{M,L},b_{M,L} \in \Gamma(L_M|_{U_M}), & a_{M,R},b_{M,R} \in \Gamma(R_M|_{U_M}), 
\end{align*}
where $U_M:=q^{-1}(U'_X)\subset V$. Then observe that
\begin{align*}
    && a_{X,L} \sim_i a_{M,L}, && b_{X,L} \sim_i b_{M,L}, && a_{X,R} \sim_i a_{M,R}, && b_{X,R} \sim_i b_{M,R},
\end{align*}
and
\begin{align*}
    & a_{M,L} \ \text{and} \ a_{M,R} \ \text{are $\circledast$-composable}, & b_{M,L} \ \text{and} \ b_{M,R} \ \text{are $\circledast$-composable}.
\end{align*}

\noindent \emph{Step 6.} Because the bracket of $i$-related sections is $i$-related, we have that 
\begin{align*}
    && [a_M,b_M] \in \Gamma\left( (L_M \circledast R_M)|_{U_M} \right) && \text{implies} && [a_X,b_X] \in \Gamma\left( (p^!(L_Y) \circledast p^!(R_Y))|_{U'_X} \right) 
\end{align*}
where
\begin{align*}
 && a_M:=a_{M,L} \circledast a_{M,R}, && b_M:=b_{M,L} \circledast b_{M,R}, && a_X:=a_{X,L} \circledast a_{X,R}, && b_X:=b_{X,L} \circledast b_{X,R}.
\end{align*}
Likewise, the bracket of $p$-related sections is $p$-related, and so
\begin{align*}
    && [a_X,b_X] \in \Gamma\left( (p^!(L_Y) \circledast p^!(R_Y))|_{U'_X} \right)  && \text{implies} && [a_Y,b_Y]|_{p(U'_X)}\in \Gamma\left( (L_Y \circledast R_Y)|_{p(U'_X)} \right).
\end{align*}
Because $p(U'_X)\subset Y$ is open and dense, we conclude that
\begin{align*}
    [a_Y,b_Y] \text{ is a smooth section of } L_Y \circledast R_Y.
\end{align*}
This concludes the proof.
\end{proof}

\begin{remark}\normalfont\label{rmk : alternative condition}
 There is another natural guess of a reduction scheme for \eqref{triangle}. Given a Dirac structure $L$, suppose that there is an involutive subbundle $C_L \subset \mathbb{T}M|_X$ such that
 \begin{align*}
  L[I] = C_L \oplus I.
 \end{align*}
 This condition ensures that $L[I]$ is a Dirac structure; hence, by Theorem \ref{thm : Dirac reduction}, $L$ reduces 
 to a Dirac structure on $Y$. However, concurrence is \underline{not} respected. For instance, for the Dirac structures $L_M$ and $R_M$ and $I=N^*X$ of Example \ref{ex : reduction does not concur}, we have
 \begin{align*}
  & L[I] = C_L \oplus N^*X, & C_L = \left\langle \tfrac{\partial}{\partial x_1}, \mathrm{d} x_2, \mathrm{d} x_3\right\rangle\\
  & R[I] = C_R \oplus N^*X, & C_R = \left\langle \mathrm{d} x_1, \tfrac{\partial}{\partial x_2}+x_1\tfrac{\partial}{\partial x_3}, x_1\mathrm{d} x_2-\mathrm{d} x_3 \right\rangle,
 \end{align*}
which are involutive. Nevertheless, the reduced (i.e., induced) Dirac structures on $X$ do not concur.
\end{remark}

\section{Reduction by kernels}\label{sec : Reduction by kernels}

\begin{minipage}{0.9\textwidth}
\begin{mdframed}[backgroundcolor=blue!5]
Here we switch our attention to the following question: given concurring Dirac structures $L$ and $R$, are there recipes to construct witnesses $E$ for both structures? Our main contribution to this question is the observation that the kernels $L_M \cap TM$ and $R_M \cap TM$ are simultaneous witnesses whenever smooth. We even conjecture that all (known) recipes are of this kind, in one way or another.
\end{mdframed} 
\end{minipage}
\vspace{0.2cm}

Let $L_M$ be a Dirac structure on $M$, and let
 \begin{align*}
  K(L_M) := L_M \cap TM
 \end{align*} 
denote its kernel. If $K(L_M)$ is a vector bundle, it defines a foliation, and if that foliation is simple, and is defined by a surjective submersion
 \begin{align*}
  p:M \to Y,
 \end{align*} 
then by Lemma \ref{lem : libermann}, $p_!(L_M)$ is a Dirac structure $L_Y$ on $Y$ (since $L_M[K(L_M)]=L_M$), and, in fact, $L_Y$ is the graph of a Poisson structure, since $L_M \cap K(L_M)=K(L_M)$ implies that $L_Y \cap TY = 0$. In short, we may say that

\begin{quotation}
{\it Every Dirac structure reduces by its kernel to a Poisson structure. }
\end{quotation}

Our next main result illustrates a rather surprising fact: that, modulo smoothness issues,

\begin{quotation}
{\it A Dirac structure reduces by the kernel of any Dirac structure with which it concurs.}
\end{quotation}

\begin{theorem}\label{thm : kernel}
 Let $L_M$ and $R_M$ be weakly concurring Dirac structures. Then
 \begin{align*}
  && L_M, && R_M, && \mathrm{Gr}(K(L_M)), && \mathrm{Gr}(K(R_M))
 \end{align*}
all concur weakly, where $\mathrm{Gr}(K(L_M))$ and $\mathrm{Gr}(K(R_M))$ are the involutive Lagrangian families corresponding to
 \begin{align*}
  & K(L_M) = L_M \cap TM, & K(R_M) = R_M \cap TM.
 \end{align*} 
\end{theorem}

\begin{proof}
Note that
\begin{align*}
 L_M[K(R_M)] & = L_M \cap K(R_M)^{\perp} + K(R_M),
\end{align*}
is involutive if
\begin{align*}
 L_M \cap K(R_M)^{\perp} = \{ u + \xi \in L_M \ | \ \xi \in K(R_M)^{\circ}\}
\end{align*}
is involutive. Now, because $L_M \cap K(R_M)^{\perp} \subset L_M \circledast R_M$, we have that
\begin{align*}
 \left[ \Gamma(L_M \cap K(R_M)^{\perp}),\Gamma(L_M \cap K(R_M)^{\perp}) \right] \subset \Gamma(L_M) \cap \Gamma(L_M \circledast R_M),
\end{align*}
and the latter lies inside $\Gamma(L_M \cap K(R_M)^{\perp})$. Thus $L_M[K(R_M)]$ is involutive, and, symmetrically, $R_M[K(L_M)]$ is also involutive. Finally,
\begin{align*}
 K(L_M)+K(R_M) = \left( L_M \circledast R_M \right) \cap TM
\end{align*}
is involutive because $L_M \circledast R_M$ and $TM$ are involutive.
\end{proof}

\begin{corollary}\label{cor : kernel diamond}
 Let $L_M$ and $R_M$ be weakly concurring Dirac structures on $M$, and suppose there are surjective submersions
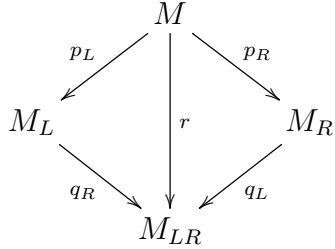
\begin{align*}
  \xymatrix{
 & M \ar[dl]_{p_L} \ar[dr]^{p_R} \ar[dd]^r & \\
 M_L \ar[dr]_{q_R} & & M_R \ar[dl]^{q_L}\\
 & M_{LR}
 }
\end{align*}
such that
 \begin{align*}
  && \ker {p_L}_* = K(L_M), && \ker {p_R}_* = K(R_M), && \ker {q_L}_* = \frac{K}{K(L_M)}, && \ker {q_R}_* = \frac{K}{K(R_M)}
 \end{align*}
where $K:=K(L_M)+K(R_M)$. Then $L_M$, $R_M$ push forward, under each of these submersions, to weakly concurring Dirac structures, and
\begin{align*}
 && {p_L}_!(L), && {p_R}_!(R), && {r}_!(L), && {r}_!(R)
\end{align*}
are Poisson.
\end{corollary}

\begin{example}\label{ex : kernel reduction}\normalfont
 Let $M=\mathbb{R}^5$ be equipped with the closed two-forms
 \begin{align*}
  & \omega_L = \mathrm{d}x_2 \wedge \mathrm{d}x_3 + \mathrm{d}x_4 \wedge \mathrm{d}x_5,
  & \omega_R = \mathrm{d}x_1 \wedge \mathrm{d}x_2 + \mathrm{d}x_3 \wedge \mathrm{d}x_4,
 \end{align*}
 and let $L_M$ and $R_M$ denote the respective graphs. They concur since
 \begin{align*}
  L_M \circledast R_M = \left\langle \tfrac{\partial}{\partial x_1},\tfrac{\partial}{\partial x_2}-\tfrac{\partial}{\partial x_4}+\mathrm{d}x_3,\tfrac{\partial}{\partial x_3}-\mathrm{d}x_2, \tfrac{\partial}{\partial x_3}+\mathrm{d}x_4,\tfrac{\partial}{\partial x_5}  \right\rangle
 \end{align*}
 is Dirac. We have a diamond
 \begin{align*}
  & 
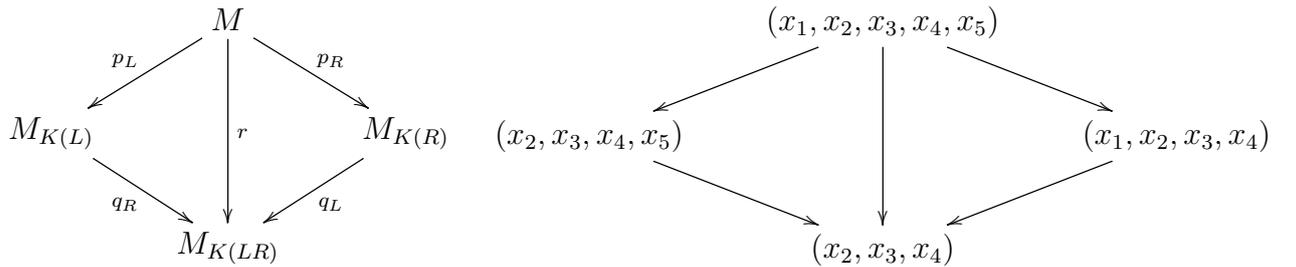
\xymatrix{
 & M \ar[dl]_{p_L} \ar[dr]^{p_R} \ar[dd]^r & \\
 M_{K(L)} \ar[dr]_{q_R} & & M_{K(R)} \ar[dl]^{q_L}\\
 & M_{K(LR)}
 }
 & \xymatrix{
 & (x_1,x_2,x_3,x_4,x_5) \ar[dl] \ar[dr] \ar[dd] & \\
 (x_2,x_3,x_4,x_5) \ar[dr] & & (x_1,x_2,x_3,x_4) \ar[dl]\\
 & (x_2,x_3,x_4)
 }
\end{align*}
and the Dirac structures implied in Corollary \ref{cor : kernel diamond} are
\begin{align*}
 & L_{M_{K(L)}} = \mathrm{Gr}(\mathrm{d}x_2 \wedge \mathrm{d}x_3 + \mathrm{d}x_4 \wedge \mathrm{d}x_5), & R_{M_{K(L)}} = \left\langle \mathrm{d}x_2,\tfrac{\partial}{\partial x_3}+\mathrm{d}x_4, \tfrac{\partial}{\partial x_4}-\mathrm{d}x_3 , \tfrac{\partial}{\partial x_5}\right\rangle\\
 & L_{M_{K(R)}} = \left\langle \tfrac{\partial}{\partial x_1},\tfrac{\partial}{\partial x_2}+\mathrm{d}x_3, \tfrac{\partial}{\partial x_3}-\mathrm{d}x_2 , \mathrm{d}x_4\right\rangle, & R_{M_{K(R)}} = \mathrm{Gr}(\mathrm{d}x_1 \wedge \mathrm{d}x_2 + \mathrm{d}x_3 \wedge \mathrm{d}x_4)\\
 & L_{M_{K(LR)}} = \mathrm{Gr}\left( \tfrac{\partial}{\partial x_3} \wedge \tfrac{\partial}{\partial x_2} \right), & R_{M_{K(LR)}} = \mathrm{Gr}\left( \tfrac{\partial}{\partial x_4} \wedge \tfrac{\partial}{\partial x_3} \right). 
\end{align*}
Clearly,
\begin{align*}
L_{M_{K(LR)}} \circledast R_{M_{K(LR)}} = \mathrm{Gr}\left( \left(\tfrac{\partial}{\partial x_4}-\tfrac{\partial}{\partial x_2}\right) \wedge \tfrac{\partial}{\partial x_3} \right). 
\end{align*}
is again Dirac.
\end{example}

\begin{remark}\normalfont
 To illustrate why concurrence is such a strong compatibility condition, suppose two closed two-forms $\omega_L,\omega_R \in \Omega^2(M)$ concur weakly, and that $\ker(\omega_R)$ is a vector bundle (and thus, a foliation). Then $\omega_L$ and $\omega_{\epsilon}:=\epsilon\omega_L+\omega_R$ concur weakly, for all $\epsilon \in \mathbb{R}$; hence by Theorem \ref{thm : kernel}, wherever smooth, $\ker(\omega_{\epsilon})$ is a witness $\omega_L$. This indicates that concurring, closed two-forms should be in short supply, and is reminiscent of the results in, e.g.,\cite{Turiel}.
\end{remark}

\section{Magri's recipe}\label{sec : Magri recipe}

In \cite{CMP}, the authors consider two commuting Poisson structures $\pi_L$ and $\pi_R$ on a manifold $M$, and take $X \subset M$ to be a leaf of $\pi_R$. Then they show that, if
 \begin{align*}
E_L:=\pi_L(\ker \pi_R)|_X \subset TM|_X
 \end{align*}
and $E_L \cap TX$ have locally constant rank, then $E_L$ is a witness for both $\pi_L$ and $\pi_R$. In fact we can also take $X=M$ \cite{Costa_Marle} and state, more symmetrically, that (under the same constant rank conditions),
 \begin{align}\label{eq : Magri original recipe}
& E_L:=\pi_L(\ker \pi_R), & E_R:=\pi_R(\ker \pi_L)
 \end{align}
are simultaneously witnesses for $\pi_L$ and $\pi_R$. We refer to \eqref{eq : Magri original recipe} as \emph{Magri's original recipe}. Below, we generalize Magri's original recipe to general weakly concurring Dirac structures. 

Consider the function of Lagrangian families
\begin{align*}
  \mathscr{N} & : \mathrm{Lag}(\mathbb{T}M) \times \mathrm{Lag}(\mathbb{T}M) \to \mathrm{Lag}(\mathbb{T}M), \\ \mathscr{N}(L,R)& :=L \star \left( R \circledast \mathcal{R}_{-1}(L) \right).
 \end{align*}

\begin{definition}\label{def: Magri's recipe}\normalfont
 We call {\bf Magri's recipe} the function
 \begin{align*}
  \mathscr{M} & : \mathrm{Lag}(\mathbb{T}M) \times \mathrm{Lag}(\mathbb{T}M) \to \mathrm{Lag}(\mathbb{T}M), \\ \mathscr{M}(L,R)& := \mathrm{Gr}(K(\mathscr{N}(L,R))).
 \end{align*}
\end{definition}

\begin{example}\normalfont
When $L$ and $R$ are, respectively, the graphs of Poisson structures $\pi_L,\pi_R \in \mathfrak{X}^2(M)$, then
\begin{align*}
 \mathscr{N}(L,R) & = \mathrm{Gr}(\pi_L) \star \mathrm{Gr}(\pi_R-\pi_L)\\
 & = \{\pi_L(\xi)+\xi \ | \ \xi \in T^*M\} \star \{\pi_R(\eta)-\pi_L(\eta)+\eta \ | \ \eta \in T^*M\}\\
 & = \{\pi_L(\xi)+\xi+\eta \ | \ \xi,\eta \in T^*M, \ \pi_L(\xi+\eta) = \pi_R(\eta)\}.
\end{align*}
Therefore $\mathscr{M}(L,R)$ is the graph of
\begin{align*}
 E_L & = \{\pi_L(\xi) \ | \ \xi \in T^*M, \ \pi_R(\xi)=0\} = \pi_L(\ker \pi_R).
\end{align*}
Hence Magri's recipe $\mathscr{M}(L,R)$ coincides with (the graph of) Magri's original recipe \eqref{eq : Magri original recipe} when the Dirac structures in question are Poisson.
\end{example}

\begin{theorem}\label{thm : new magri recipe}
 Let $L$ and $R$ be weakly concurring Dirac structures on $M$. Then the Magri recipes
 \begin{align*}
  & \mathscr{M}(L,R), & \mathscr{M}(R,L)
 \end{align*}
 are involutive, and
 \begin{align*}
  && L, && R, && \mathscr{M}(L,R), && \mathscr{M}(R,L)
 \end{align*}
concur weakly pairwise.
\end{theorem}
\begin{proof}
Let us first observe that, because $L$ and $R$ concur weakly, we have that $R \circledast \mathcal{R}_{-1}(L)$ is involutive, and therefore so is the tangent product $\mathscr{N}(L,R):= L \star (R \circledast \mathcal{R}_{-1}(L))$. Therefore $E_L:=\mathscr{N}(L,R) \cap TM$ is involutive; likewise, $E_R:=\mathscr{N}(R,L) \cap TM$ is involutive.

Let $U \subset M$ be the maximal open set on which
 \begin{align*}
  && K(L), && K(R), && K(L)+K(R), && \mathscr{N}(L,R), && E_L, && \mathscr{N}(R,L), && E_R
 \end{align*}
are vector bundles. Then $U$ is dense in $M$. It thus suffices to check that the theorem holds in an open neighborhood $V \subset U$ of every point $x \in U$. We may thus assume that we are in the setting of Corollary \ref{cor : kernel diamond} --- that is, that there exists a commutative diagram of surjective submersions with connected fibres
\begin{align*}
  \xymatrix{
 & M \ar[dl]_{p_L} \ar[dr]^{p_R} \ar[dd]^r & \\
 M_L \ar[dr]_{q_R} & & M_R \ar[dl]^{q_L}\\
 & M_{LR}
 }
\end{align*}
in which
\begin{align*}
  && \ker {p_L}_* = K(L), && \ker {p_R}_* = K(R), && \ker {q_L}_* = K/K(R), && \ker {q_R}_* = K/K(L),
 \end{align*}
where $K:=K(L)+K(R)$, and there exist Poisson structures
\begin{align*}
 && \pi_L \in \mathfrak{X}^2(M_L), && \pi_R \in \mathfrak{X}^2(M_R),
 && \nu_L \in \mathfrak{X}^2(M_{LR}), && \nu_R \in \mathfrak{X}^2(M_{LR}),
\end{align*}
such that
\begin{align*}
 && L={p_L}^!\mathrm{Gr}(\pi_L), && R={p_R}^!\mathrm{Gr}(\pi_R),
 && r_!(L)=\mathrm{Gr}(\nu_L), && r_!(R)=\mathrm{Gr}(\nu_R).
\end{align*}
We may thus write
\begin{align*}
 & L = \mathcal{R}_{\Pi_L}\mathrm{Gr}(K(L)), & R = \mathcal{R}_{\Pi_R}\mathrm{Gr}(K(R)), 
\end{align*}
for two bivectors $\Pi_L,\Pi_R \in \mathfrak{X}^2(M)$ for which
\begin{align*}
 & \pi_L(\xi) = (p_L)_*\Pi_L((p_L)^*\xi), & \xi \in T^*M_L,\\
 & \pi_R(\eta) = (p_R)_*\Pi_R((p_R)^*\eta), & \eta \in T^*M_R.
\end{align*}

\noindent \emph{Step 1. Description of $E_L$ and $E_R$.} Because
\begin{align*}
 L & = \left\{ u+\Pi_L(\xi) + \xi \ | \ u \in K(L), \ \xi \in K(L)^{\circ} \right\}\\
 R \circledast \mathcal{R}_{-1}(L) & = \left\{ v + \Pi_R(\eta) - \Pi_L(\eta) + \eta \ | \ v \in K, \ \eta \in K^{\circ} \right\}
\end{align*}
we see that
\begin{align*}
 \mathscr{N}(L,R) & = \left\{ u+\Pi_L(\xi) + \xi + \eta \ \Bigg| \ 
 \begin{aligned}
  u \in K(L), & \ v \in K,\\
  \xi \in K(L)^{\circ}, & \ \eta \in K^{\circ},\\
  u+\Pi_L(\xi+\eta) & = v + \Pi_R(\eta)
 \end{aligned}
 \right\}.
\end{align*}
Therefore
\begin{align*}
 & E_L = K(L) + \Pi_L(H_L), & H_L = \left\{ \xi \in K^{\circ} \ | \ \Pi_R(\xi) \in K\right\}.
\end{align*}
Observe that
\begin{align*}
H_L = \left( R \circledast \mathrm{Gr}(K) \right) \cap T^*M,
\end{align*}
and therefore
\begin{align*}
H_L^{\circ} = \mathrm{pr}_T\left( R \circledast \mathrm{Gr}(K) \right) = K + \Pi_R(K^{\circ}).
\end{align*}
This implies that
\begin{align*}
E_L^{\circ} & = K(L)^{\circ} \cap \Pi_L^{-1}(H_L^{\circ})\\
& = \left\{ \xi \in K(L)^{\circ} \ | \ \Pi_L(\xi) \in K + \Pi_R(K^{\circ}) \right\}
\end{align*}
Likewise
\begin{align*}
 & E_R = K(R) + \Pi_R(H_R), & E_R^{\circ} & = \left\{ \xi \in K(R)^{\circ} \ | \ \Pi_R(\xi) \in K + \Pi_L(K^{\circ}) \right\}.
\end{align*}

\noindent \emph{Step 2. $L$ and $\mathscr{M}(L,R)$ concur weakly.} Because
\begin{align*}
 L \circledast \mathscr{M}(L,R) & = L \cap E_L^{\perp} + E_L
\end{align*}
and $E_L$ is involutive, in order to show that $L$ and $\mathscr{M}(L,R)$ concur weakly, it suffices to show that 
\begin{align}\label{eq : LEL}
 L \cap E_L^{\perp} & = K(L) + \mathcal{R}_{\Pi_L}(E_L^{\circ})
\end{align}
is involutive. Because $L$ is involutive, we have that
\begin{enumerate}[La)]
 \item $\Pi_L[v,\xi] - [v,\Pi_L(\xi)] \in \Gamma(K(L))$,
 \item $[\xi,\eta]^{\Pi_L} \in \Gamma(K(L)^{\circ})$;
 \item $\Pi_L[\xi,\eta]^{\Pi_L} - [\Pi_L(\xi),\Pi_L(\eta)]\in \Gamma(K(L))$
\end{enumerate}
for all $v \in \Gamma(K(L))$ and all $\xi,\eta \in \Gamma(K(L)^{\circ})$. Likewise, $R$ involutive implies
\begin{enumerate}[Ra)]
 \item $\Pi_R[v,\xi] - [v,\Pi_R(\xi)] \in \Gamma(K(R))$,
 \item $[\xi,\eta]^{\Pi_R} \in \Gamma(K(R)^{\circ})$;
 \item $\Pi_R[\xi,\eta]^{\Pi_R} - [\Pi_R(\xi),\Pi_R(\eta)]\in \Gamma(K(R))$
\end{enumerate}
Moreover, because
\begin{align*}
 L \circledast \mathrm{Gr}(K) & = K + \mathcal{R}_{\Pi_L}(K^{\circ})
\end{align*}
is involutive by Theorem \ref{thm : kernel}, we have that
\begin{enumerate}[LKa)]
 \item $\Pi_L[v,\xi] \equiv [v,\Pi_L(\xi)]$,
 \item $[\xi,\eta]^{\Pi_L} \in \Gamma(K^{\circ})$;
 \item $\Pi_L[\xi,\eta]^{\Pi_L} \equiv [\Pi_L(\xi),\Pi_L(\eta)]$
\end{enumerate}
for all $v \in \Gamma(K)$ and $\xi,\eta \in \Gamma(K^{\circ})$, where we write $\equiv$ for equality up to $\Gamma(K)$. Likewise, 
\begin{align*}
 R \circledast \mathrm{Gr}(K) & = K + \mathcal{R}_{\Pi_R}(K^{\circ})
\end{align*}
involutive implies that
\begin{enumerate}[RKa)]
 \item $\Pi_R[v,\xi] \equiv [v,\Pi_R(\xi)]$,
 \item $[\xi,\eta]^{\Pi_R} \in \Gamma(K^{\circ})$;
 \item $\Pi_R[\xi,\eta]^{\Pi_R} \equiv [\Pi_R(\xi),\Pi_R(\eta)]$
\end{enumerate}
We show that
\begin{align}\label{eq : LEL1}
 \left[ \Gamma(K(L)), \Gamma\left(\mathcal{R}_{\Pi_L}(E_L^{\circ})\right) \right] \subset \Gamma(K(L)) + \Gamma\left(\mathcal{R}_{\Pi_L}(E_L^{\circ})\right).
\end{align}
Let $v \in K(L)$ and $\xi \in \Gamma(E_L^{\circ})$. Then there exists $\xi' \in \Gamma(K^{\circ})$, such that $\Pi_L(\xi) \equiv \Pi_R(\xi')$. Note that $[v,\xi] \in \Gamma(K(L)^{\circ})$ and $[v,\xi'] \in \Gamma(K^{\circ})$, and because
\begin{align*}
 \Pi_L[v,\xi] \equiv [v,\Pi_L(\xi)] \equiv [v,\Pi_R(\xi')] \equiv \Pi_R[v,\xi']
\end{align*}
by La) and RKa), we deduce that, in fact, $[v,\xi] \in \Gamma(E_L^{\circ})$. Therefore
\begin{align*}
 [v,\Pi_L(\xi)+\xi] \equiv \Pi_L[v,\xi]+[v,\xi] \in \Gamma\left(\mathcal{R}_{\Pi_L}(E_L^{\circ})\right).
\end{align*}
Hence \eqref{eq : LEL1} holds true. We show next that 
\begin{align}\label{eq : LEL2}
 \left[ \Gamma\left(\mathcal{R}_{\Pi_L}(E_L^{\circ})\right), \Gamma\left(\mathcal{R}_{\Pi_L}(E_L^{\circ})\right) \right] \subset \Gamma(K(L)) + \Gamma\left(\mathcal{R}_{\Pi_L}(E_L^{\circ})\right).
\end{align}
Let $\xi,\eta \in \Gamma(E_L^{\circ})$. Observe that, since $E_L^{\circ} \subset K(L)^{\circ}$, we have that $[\xi,\eta]^{\Pi_L} \in \Gamma(K(L)^{\circ})$ by Lb). Write now 
\begin{align*}
 && \Pi_L(\xi) = \Pi_R(\xi')+u, && \Pi_L(\eta) = \Pi_R(\eta')+v, && \xi',\eta'\in \Gamma(K^{\circ}), && u,v \in \Gamma(K).
\end{align*}
Then
\begin{align*}
 \Pi_L[\xi,\eta]^{\Pi_L} & \equiv [\Pi_L(\xi),\Pi_L(\eta)] \equiv [\Pi_R(\xi')+u,\Pi_R(\eta')+v] \\
 & \equiv [\Pi_R(\xi'),\Pi_R(\eta')]+[u,\Pi_R(\eta')]+[\Pi_R(\xi'),v]+[u,v]\\
 & \equiv \Pi_R\left([\xi',\eta']^{\Pi_R}+[u,\eta']+[\xi',v]\right)
\end{align*}
shows that $[\xi,\eta]^{\Pi_L} \in \Gamma(E_L^{\circ})$, since $[\xi',\eta']^{\Pi_R} \in \Gamma(K^{\circ})$ by RKb). Therefore
\begin{align*}
 [\Pi_L(\xi)+\xi,\Pi_L(\eta)+\eta] \equiv \Pi_L[\xi,\eta]^{\Pi_L}+[\xi,\eta]^{\Pi_L}
\end{align*}
shows that \eqref{eq : LEL2} holds true.\\

\noindent \emph{Step 3. $R$ and $\mathscr{M}(L,R)$ concur weakly.} As before,
\begin{align*}
 R \circledast \mathscr{M}(L,R) & = R \cap E_L^{\perp} + E_L
\end{align*}
is involutive if
\begin{align}\label{eq : REL}
 R \cap E_L^{\perp} & = K(R) + \mathcal{R}_{\Pi_R}(K(R)^{\circ} \cap E_L^{\circ})
\end{align}
is involutive. Let $v \in \Gamma(K(R))$ and $\xi \in \Gamma(K(R)^{\circ} \cap E_L^{\circ})$. Because $\xi \in \Gamma(E_L^{\circ})$, 
\begin{align*}
 & \Pi_L(\xi) = \Pi_R(\xi')+u, & \xi'\in \Gamma(K^{\circ}), \ u \in \Gamma(K).
\end{align*}
Then $[v,\xi] \in \Gamma(K(R)^{\circ})$, and because
\begin{align*}
 \Pi_L[v,\xi] \equiv [v,\Pi_L(\xi)] \equiv [v,\Pi_R(\xi')+u] \equiv \Pi_R[v,\xi']
\end{align*}
we see that $[v,\xi] \in \Gamma(K(R)^{\circ} \cap E_L^{\circ})$. Therefore
\begin{align}\label{eq : REL1}
 [v,\Pi_R(\xi)+\xi] & \equiv \Pi_R[v,\xi]+[v,\xi] \in \Gamma\left( \mathcal{R}_{\Pi_R}(K(R)^{\circ} \cap E_L^{\circ}) \right).
\end{align}
Let next $\eta \in \Gamma(K^{\circ} \cap E_L^{\circ})$. Then $[\xi,\eta]^{\Pi_R} \in \Gamma(K^{\circ})$ by RKb). Because
\begin{align*}
 L \circledast R \circledast \mathrm{Gr}(K) = K + \mathcal{R}_{\Pi_L+\Pi_R}(K^{\circ})
\end{align*}
is involutive, we have that
\begin{align}\tag{LRK}
 \Pi_L[\xi,\eta]^{\Pi_R}+\Pi_R[\xi,\eta]^{\Pi_L} \equiv [\Pi_L(\xi),\Pi_R(\eta)] + [\Pi_R(\xi),\Pi_L(\eta)]
\end{align}
for all $\xi,\eta \in  \Gamma(K^{\circ})$. Therefore, if $\xi,\eta \in \Gamma(K^{\circ} \cap E_L^{\circ})$, that is
\begin{align*}
 && \Pi_L(\xi) = \Pi_R(\xi')+u, && \Pi_L(\eta) = \Pi_R(\eta')+v, && \xi',\eta'\in \Gamma(K^{\circ}), && u,v \in \Gamma(K),
\end{align*}
then
\begin{align*}
 \Pi_L[\xi,\eta]^{\Pi_R} & \equiv [\Pi_L(\xi),\Pi_R(\eta)] + [\Pi_R(\xi),\Pi_L(\eta)] - \Pi_R[\xi,\eta]^{\Pi_L} \\
 & \equiv [\Pi_R(\xi')+u,\Pi_R(\eta)] + [\Pi_R(\xi),\Pi_R(\eta')+v] - \Pi_R[\xi,\eta]^{\Pi_L}\\
 & \equiv \Pi_R\left([\xi',\eta]^{\Pi_R} + [u,\eta] + [\xi,\eta']^{\Pi_R} + [\xi,v] - [\xi,\eta]^{\Pi_L}\right)
\end{align*}
shows that $[\xi,\eta]^{\Pi_R} \in \Gamma(K^{\circ} \cap E_L^{\circ})$. Therefore
\begin{align}\label{eq : REL2}
 [\Pi_R(\xi)+\xi,\Pi_R(\eta)+\eta] & \equiv \Pi_R[\xi,\eta]^{\Pi_R} + [\xi,\eta]^{\Pi_R} \in \Gamma\left(\mathcal{R}_{\Pi_R}(K^{\circ} \cap E_L^{\circ})\right)
\end{align}
Then observe that \eqref{eq : REL1} and \eqref{eq : REL2} imply that \eqref{eq : REL} is involutive.\\

\noindent \emph{Step 4. $L$ and $R$ concur weakly with $\mathscr{M}(R,L)$.} Just reverse the roles of $L$ and $R$ in Steps 2 and 3.\\

\noindent \emph{Step 5. $\mathscr{M}(L,R)$ and $\mathscr{M}(R,L)$ concur weakly.} 
We show that
\begin{align*}
 E:=E_L+E_R \subset TM
\end{align*}
is involutive. By Theorem \ref{thm : kernel}, there exist Poisson structures $\nu_L,\nu_R \in \mathfrak{X}^2(M_{LR})$, such that
\begin{align*}
 & r_!(L) = \mathrm{Gr}(\nu_L), & r_!(R) = \mathrm{Gr}(\nu_R).
\end{align*}
We can then write
\begin{align*}
 & H_L = r^*(\ker \nu_R), & H_R = r^*(\ker \nu_L),
\end{align*}
and therefore
\begin{align*}
 E & = E_L +E_R = K + \Pi_Lr^*(\ker \nu_R) + \Pi_Rr^*(\ker \nu_L) \\ & = r_*^{-1}\left( \nu_L(\ker(\nu_R))+\nu_R(\ker(\nu_L)) \right).
\end{align*}
So it suffices to show that $\nu_L(\ker(\nu_R))+\nu_R(\ker(\nu_L))$ is involutive. Because $\nu_L$ and $\nu_R$ commute by Theorem \ref{thm : reduction of weakly concurring concur weakly}, we have that
\begin{align}\label{eq : commuting Poisson}
 \nu_L[\xi,\eta]^{\nu_R}+\nu_R[\xi,\eta]^{\nu_L} = [\nu_L(\xi),\nu_R(\eta)]+[\nu_R(\xi),\nu_L(\eta)]
\end{align}
for all $\xi,\eta \in \Omega^1(M_{LR})$, and therefore
\begin{align*}
 & \left[\Gamma(\ker \nu_R),\Gamma(\ker \nu_R)\right]^{\nu_L} \subset \Gamma(\ker \nu_R).
\end{align*}
This implies that, for all $\xi,\eta \in \Gamma(\ker \nu_R)$,
\begin{align*}
 [\nu_L(\xi),\nu_L(\eta)] = \nu_L[\xi,\eta]^{\nu_L},
\end{align*}
and so $\nu_L(\ker \nu_R)$ is involutive. By symmetry, also $\nu_R(\ker \nu_L)$ is involutive. Again by \eqref{eq : commuting Poisson}, if $\xi \in \Gamma(\ker \nu_R)$ and $\eta \in \Gamma(\ker \nu_L)$, then
\begin{align*}
 & [\xi,\eta]^{\nu_L} \in \Gamma(\ker \nu_L),
 & [\xi,\eta]^{\nu_R} \in \Gamma(\ker \nu_R),
\end{align*}
and so
\begin{align*}
 [\nu_L(\xi),\nu_R(\eta)] = \nu_L[\xi,\eta]^{\nu_R}+\nu_R[\xi,\eta]^{\nu_L}
\end{align*}
lies in $\nu_L(\ker(\nu_R))+\nu_R(\ker(\nu_L))$. This concludes the proof.

\end{proof}

\begin{corollary}\label{cor : Magri diamond}
 Let $L_M$ and $R_M$ be weakly concurring Dirac structures on $M$, and suppose there are surjective submersions
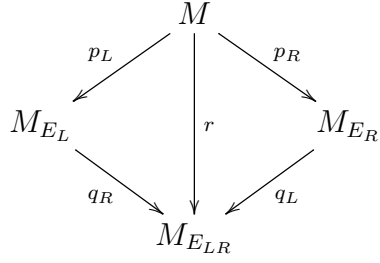
\begin{align*}
  \xymatrix{
 & M \ar[dl]_{p_L} \ar[dr]^{p_R} \ar[dd]^r & \\
 M_{E_L} \ar[dr]_{q_R} & & M_{E_R} \ar[dl]^{q_L}\\
 & M_{E_{LR}}
 }
\end{align*}
such that
 \begin{align*}
  & \ker {p_L}_* = \mathscr{M}(L,R) \cap TM, & \ker {p_R}_* = \mathscr{M}(R,L) \cap TM, \\
  & \ker {q_L}_* = \frac{(\mathscr{M}(L,R) \circledast \mathscr{M}(R,L)) \cap TM}{\mathscr{M}(L,R) \cap TM}, & \ker {q_R}_* = \frac{(\mathscr{M}(L,R) \circledast \mathscr{M}(R,L)) \cap TM}{\mathscr{M}(R,L) \cap TM}.
 \end{align*}
Then $L$, $R$ push forward, under each of these submersions, to weakly concurring Dirac structures, and
\begin{align*}
 && {p_L}_!(L), && {p_R}_!(R), && {r}_!(L), && {r}_!(R)
\end{align*}
are Poisson.
\end{corollary}

\begin{example}\label{ex : Magri diamond}\normalfont
 Consider the Dirac structures $L_M$ and $R_M$ on $M=\mathbb{R}^5(x_1,x_2,x_3,x_4,x_5)$ of Example \ref{ex : kernel reduction}. Observe that
 \begin{align*}
  & \mathscr{N}(L_M,R_M) = \langle \tfrac{\partial}{\partial x_1}, \tfrac{\partial}{\partial x_2}+\tfrac{\partial}{\partial x_4} + \mathrm{d}x_5, \tfrac{\partial}{\partial x_3}, \mathrm{d}x_2-\mathrm{d}x_4, \tfrac{\partial}{\partial x_5} - \mathrm{d}x_4 \rangle
 \end{align*}
 has kernel $E_L = \langle \tfrac{\partial}{\partial x_1}, \tfrac{\partial}{\partial x_3} \rangle$. Likewise, $E_R = \langle \tfrac{\partial}{\partial x_3}, \tfrac{\partial}{\partial x_5} \rangle$. We thus have a diamond
 \begin{align*}
  & 
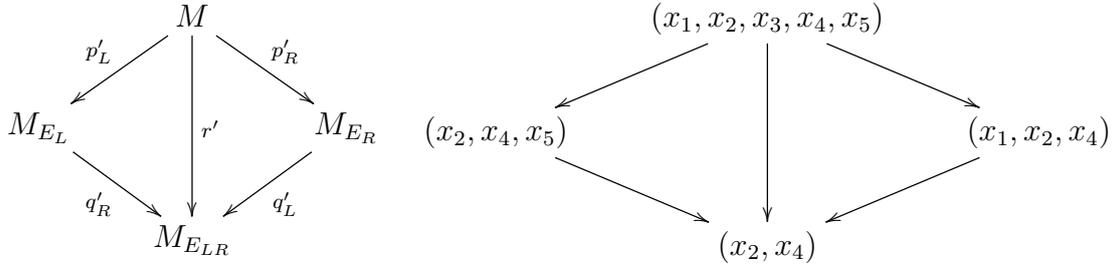
\xymatrix{
 & M \ar[dl]_{p'_L} \ar[dr]^{p'_R} \ar[dd]^{r'} & \\
 M_{E_L} \ar[dr]_{q'_R} & & M_{E_R} \ar[dl]^{q'_L}\\
 & M_{E_{LR}}
 }
 & \xymatrix{
 & (x_1,x_2,x_3,x_4,x_5) \ar[dl] \ar[dr] \ar[dd] & \\
 (x_2,x_4,x_5) \ar[dr] & & (x_1,x_2,x_4) \ar[dl]\\
 & (x_2,x_4)
 }
\end{align*}
which induce the Dirac structures
\begin{align*}
 & L_{M_{E_L}} = \mathrm{Gr}(\tfrac{\partial}{\partial x_5} \wedge \tfrac{\partial}{\partial x_4}), & R_{M_{E_L}} = \mathrm{Gr}(\tfrac{\partial}{\partial x_5})\\
 & L_{M_{E_R}} = \mathrm{Gr}(\tfrac{\partial}{\partial x_1}), 
 & R_{M_{E_R}} = \mathrm{Gr}(\tfrac{\partial}{\partial x_2} \wedge \tfrac{\partial}{\partial x_1})\\
 & L_{M_{E_{LR}}} = T^*M_{E_{LR}}, & R_{M_{E_{LR}}} = T^*M_{E_{LR}}. 
\end{align*}
\end{example}

\section{Further comments and examples}

\subsection{Moment map condition}
\vspace{0.2cm}\phantom{}

\begin{minipage}{0.9\textwidth}
\begin{mdframed}[backgroundcolor=blue!5]
The classical example of witnesses arises in Hamiltonian spaces \cite[Example B]{MR}. We recall that setting and illustrate some of our constructions in this case.
\end{mdframed} 
\end{minipage}
\vspace{0.2cm}

Let a Lie group $G$ act on a manifold $M$, and let $L_M$ be a Dirac structure on $M$. We say that the action is {\bf Dirac} if $G$ acts by Dirac automorphisms: for each $g \in G$, the induced automorphism
\begin{align*}
& g_*: \mathbb{T}M \diffto \mathbb{T}M, & g_*(u+\alpha) = g_*(u)+(g^{-1})^*(\alpha)
\end{align*}
maps $L_M$ to itself. If $G$ is connected, the condition that $G$ act by Dirac automorphisms is equivalent to the requirement that
 \begin{align}\label{eq : Diracaction}
     & [\mathrm{a}(v),\Gamma(L_M)] \subset \Gamma(L_M), & v \in \mathfrak{g},
 \end{align}
where $\mathrm{a} : \mathfrak{g} \to \mathfrak{X}(M)$ stands for the infinitesimal action. 

A Dirac action $G \curvearrowright (M,L_M)$ is {\bf Hamiltonian} if it is equipped with an ${\rm Ad}^*$-equivariant map
\begin{align*}
    & \mu : M \to \mathfrak{g}^*, & \mu(g \cdot x) = \mathrm{Ad}^*_{g^{-1}}\mu(x),
\end{align*}
with the property that
\begin{align}\label{eq : moment map condition}
    & \mathrm{a}(v) + \mathrm{d}\langle \mu,v\rangle \ \in \Gamma(L_M), \quad v \in \mathfrak{g}.
\end{align}
We refer to $G \curvearrowright (M,L_M) \stackrel{\mu}{\rmap}\mathfrak{g}^*$ as a {\bf Hamiltonian space} \cite{BF}.

\begin{definition}\normalfont\label{def : good value}
Let $G \curvearrowright (M,L_M) \stackrel{\mu}{\rmap}\mathfrak{g}^*$ be a Hamiltonian space in which the action is free. We will call an element $\alpha \in \mathfrak{g}^*$ {\bf good} if:
\begin{enumerate}[a)]
 \item $\mu$ meets $\alpha$ cleanly --- that is, $X = \mu^{-1}(\alpha)$ is a smooth submanifold, and
 \begin{align*}
  & T_xX = \{ u \in T_xM \ | \ \mu_*(u) = 0\}, & x \in X;
 \end{align*}
 \item The isotropy subgroup $G_{\alpha} = \{ g \in G \ | \ \mathrm{Ad}^*_{g^{-1}}(\alpha)=\alpha\}$ acts properly on $X$;
 \item The map $j:L_M|_X \to \mathfrak{g}^*$ given by
 \begin{align*}
  & \langle j(u+\xi),v\rangle:=\langle \xi,\mathrm{a}(v)\rangle, & v \in \mathfrak{g}
 \end{align*}
has locally constant rank.
\end{enumerate}
\end{definition}

\begin{lemma}\label{lem : witness for good value}
If $\alpha \in \mathfrak{g}^*$ is a good value of a free Hamiltonian space
\begin{align*}
 G \curvearrowright (M,L_M) \stackrel{\mu}{\rmap} \mathfrak{g}^*,
\end{align*}
then $E_G \subset TM|_X$ is a witness for $L_M$, where
\begin{align*}
 & X = \mu^{-1}(\alpha), & E_G = \{\mathrm{a}(v)_x \ | \ v \in \mathfrak{g}, \ x \in X\}.
\end{align*}
\end{lemma}
\begin{proof}
 The action being free on $M$ implies that $\mathrm{a}_x:\mathfrak{g} \to T_xM$ is injective for each $x \in M$; therefore $E_G=\mathrm{a}(\mathfrak{g})|_X$ is a vector subbundle of $TM|_X$, and this is an adapted subbundle because
 \begin{align*}
  E_G \cap TX & = \{ \mathrm{a}(v) \ | \ \mu_*\mathrm{a}(v)=\pi_{\mathfrak{g}}^{\sharp}(v) = 0\} = \mathrm{a}(\mathfrak{g}_{\alpha})|_X.
 \end{align*}
Note that
\begin{align*}
 L_M \cap E_G^{\perp} & = \{ u+\xi \in L_M \ | \ \langle \xi,\mathrm{a}(v)\rangle = 0, \ v \in \mathfrak{g} \}\\
 & = \{ u+\xi \in L_M \ | \ \langle u,\mu^*(v) \rangle = 0, \ v \in \mathfrak{g} \}\\
 & = L_M \cap (N^*X)^{\perp},
\end{align*}
where in the second equality we used that $\mathrm{a}(v)+\mu^*(v) \in L_M$ implies that
\begin{align*}
 \langle \xi,\mathrm{a}(v)\rangle + \langle u,\mu^*(v) \rangle = \langle u+\xi,\mathrm{a}(v)+\mu^*(v) \rangle = 0
\end{align*}
for all $u+\xi \in L_M$. Condition \hyperref[Wit3]{Wit3)} follows immediately, and also condition \hyperref[Wit1]{Wit1)}, since
 \begin{align*}
 L_M \cap I^{\perp} & = \left(L_M \cap (N^*X)^{\perp}\right) \cap F^{\perp} = \left(L_M \cap E^{\perp}\right) \cap F^{\perp} = L_M \cap E^{\perp},
\end{align*}
(and this implies that also $L_M \cap I$ and $L_M[I]=L_M \cap I^{\perp}+I$ are vector bundles.)

On the other hand, note that $F$ is the foliation by orbits of $G_{\alpha}$, and that if $u+\xi \in \Gamma(L_M \cap E_G^{\perp})$ and $v \in \mathfrak{g}_{\alpha}$, then
\begin{align*}
 [\mathrm{a}(v),u+\xi] & = [\mathrm{a}(v),u]+[\mathrm{a}(v),\xi]
\end{align*}
again lies in $\Gamma(L_M \cap E_G^{\perp})$, since the action is Dirac and
\begin{align*}
 \langle [\mathrm{a}(v),\xi],\mathrm{a}(w)\rangle = \mathscr{L}_{\mathrm{a}(v)}\langle \xi,\mathrm{a}(w)\rangle - \langle \xi,\mathrm{a}([v,w])\rangle = 0
\end{align*}
for all $w \in \mathfrak{g}_{\alpha}$ --- that is,
\begin{align}
  \left[ \Gamma(F) ,\Gamma(L_M \cap E^{\perp}) \right] \subset \Gamma(F) + \Gamma(L_M \cap E^{\perp}).
\end{align}
Now, because
\begin{align*}
 L_M[E] \cap (N^*X)^{\perp} = \left( L_M \cap E^{\perp} + E \right) \cap (N^*X)^{\perp} = L_M \cap E^{\perp} + F,
\end{align*}
and 
\begin{align*}
 \left[ \Gamma(L_M \cap E^{\perp}) ,\Gamma(L_M \cap E^{\perp}) \right] & = \left[ \Gamma(L_M \cap (N^*X)^{\perp}) ,\Gamma(L_M \cap (N^*X)^{\perp}) \right] \\
 & = \Gamma(L_M \cap (N^*X)^{\perp}) = \Gamma(L_M \cap E^{\perp})
\end{align*}
because the bracket of any two vector fields of $M$ which are tangent to $X$ is again tangent to $X$, we deduce that \hyperref[Wit2]{Wit2)} is satisfied. Hence $E_G$ is a witness for $L_M$.
\end{proof}

\begin{remark}\label{rmk : regularity of brahic fernandes}\normalfont
    Stronger versions of conditions a) and c) above are ensured by the notion of regularity introduced in \cite{BF}, where one requires that $\alpha\in \mathfrak{g}^*$ be a regular value for the restriction of the momentum map $\mu|_S : S \to \mathfrak{g}^*$ to every leaf $S$ of $L_M$. This condition in fact implies  that $\mu$ meets $\alpha$ transversely, and that $j:L_M|_X\rightarrow \mathfrak{g}^*$ is fiberwise surjective.
\end{remark}

\begin{remark}\normalfont
In the setting of Poisson manifolds, a Hamiltonian $G$-space whose infinitesimal action is free is necessarily regular in the sense of Remark \ref{rmk : regularity of brahic fernandes}. For general Dirac structures, this is no longer the case, as the following example illustrates.
\end{remark}

\begin{example}\label{examplegood}\normalfont
Consider $M=\mathbb{R}^4$, equipped with the action
\begin{align*}
 & G:=\mathbb{R}^2 \curvearrowright M, (a,b) \cdot (x_1,x_2,x_3,x_4):=(x_1+a,x_2,x_3+a,x_4+b)
\end{align*}
and the Dirac structure $L_M$ corresponding to the graph of the closed two-form $\omega_M=\mathrm{d}x_1 \wedge \mathrm{d}x_2$. Then
\begin{align*}
 & G \curvearrowright (M,L_M) \stackrel{\mu}{\rmap} \mathfrak{g}^*\simeq \mathbb{R}^2, & \mu(x_1,x_2,x_3,x_4) = (x_2,0)
\end{align*}
is a Hamiltonian $G$-space. Let $X:=\mu^{-1}(0)$, and note that $0 \in \mathfrak{g}^*$ is a good value which is not regular. Thus
\begin{align*}
 E_G|_X = E_G \cap TX = \left\langle \tfrac{\partial}{\partial x_1}+\tfrac{\partial}{\partial x_3},\tfrac{\partial}{\partial x_4} \right\rangle
\end{align*}
is a witness of the reduction of $i^!(L_M) = TX$ to the tangent bundle of a line.
\end{example}

\subsection{Dirac-Nijenhuis structures}

An endomorphism $N_M:TM \to TM$ (i.e. a linear map covering the identity) gives rise to an $N_M$-twisted bracket $[\cdot,\cdot]^{N_M} : \mathfrak{X}(M) \times \mathfrak{X}(M) \to \mathfrak{X}(M)$, given by
\begin{align*}
 [u,v]^{N_M}:=[N_M(u),v]+[u,N_M(v)]-N_M[u,v].
\end{align*}
We say that $N_M$ is {\bf Nijenhuis} if
\begin{align*}
 & N_M[u,v]^{N_M}=[N_M(u),N_M(v)], & u,v \in \mathfrak{X}(M).
\end{align*}
A smooth map $\phi : M \to P$ {\bf relates} two Nijenhuis maps $N_M : TM \to TM$ and $N_P : TP \to TP$ if
\begin{align*}
 && u_M \in \mathfrak{X}(M), && u_P \in \mathfrak{X}(P), && u_M \sim_{\phi} u_P
\end{align*}
implies that $N_M(u_M) \sim_{\phi} N_P(u_P)$.

If $L_M$ is a Dirac structure on $M$, we say that $(L_M,N_M)$ is {\bf Dirac-Nijenhuis} \cite[Section 3.3]{BDN} if
\begin{align*}
 & (N_M,N_M^*)(L) \subset L, & \{ \Gamma(\mathbb{T}M), \Gamma(L_M) \}_{N_M} \subset \Gamma(L_M),
\end{align*}
where $\{\cdot,\cdot\}_{N_M}:\Gamma(\mathbb{T}M) \times \Gamma(\mathbb{T}M) \to \Gamma(\mathbb{T}M)$ denotes
\begin{align}\label{eq : bracket of prolongation of a}
 \{ u + \xi,v+\eta\}_{N_M} & = [v,N_M(u)]-N_M[v,u] + [\eta,N_M(u)]-[N_M^*(\eta),u].
\end{align}

\begin{example}\normalfont
 When $L_M$ is the graph of a Poisson structure $\pi_M$, the condition that $(L_M,N_M)$ be Dirac-Nijenhuis is equivalent to
 \begin{align*}
  & N_M\pi_M = \pi_M N_M^*, & N_M^*[\xi,\eta]^{\pi_M} = [N_M^*(\xi),\pi_M(\eta)]+[\pi_M(\xi),N_M^*(\eta)]
 \end{align*}
 for all $\xi,\eta \in \Omega^1(M)$ --- that is, that $(\pi_M,N_M)$ be a {\bf $PN$-structure}. 
\end{example}

\begin{example}\normalfont
 When $L_M$ is the graph of a closed two-form $\omega_M$, the condition that $(L_M,N_M)$ be Dirac-Nijenhuis is equivalent to
 \begin{align*}
  & \omega N_M = N_M^* \omega, & \mathrm{d}(\omega N_M) = 0;
 \end{align*}
 that is, that $(\omega_M,N_M)$ be a {\bf $\Omega N$-structure}. 
\end{example}

\begin{lemma}\label{lem : MR for DN}
 Assume that the diagram \eqref{triangle} relates Nijenhuis endomorphisms\footnote{As explained in \cite[Section 3]{Fernandes_master} (see also \cite[Theorem 2.1]{Vaisman96}), the reducibility of the Nijenhuis map is independent of that of the Dirac structures. Observe also that the proof of the analogous result for the \emph{weak} Dirac-Nijenhuis structures of \cite{FMT_Nijenhuis} is automatic.}
 \begin{align*}
  && N_M :TM \to TM, 
  && N_X :TX \to TX,
  && N_Y :TY \to TY.
 \end{align*}
If $(L_M,N_M)$ is a Dirac-Nijenhuis structure, and $L_M$ has a Dirac reduction to $L_Y$ on $Y$, then also $(L_Y,N_Y)$ is a Dirac-Nijenhuis structure.
\end{lemma}
\begin{proof}
Because $i^!(L_M)$ is smooth in an open, dense set $U \subset X$, the result follows by density from juxtaposing \cite[Proposition 4.1]{BDN} and \cite[Proposition 4.4]{BDN}.
\end{proof}

A Dirac-Nijenhuis structure $(L_M,N_M)$ gives rise to Lagrangian families
\begin{align*}
 \mathfrak{L}_{N_M^n}(L_M) & = (N_M^n,\mathrm{id})(L_M)+L_M \cap TM,\\
 \mathfrak{R}_{N_M^n}(L_M) & = (\mathrm{id},(N_M^*)^n)(L_M)+L_M \cap T^*M
\end{align*}
of {\bf left-} and {\bf right shifts} (see \cite[Section 9]{FMT_Nijenhuis} and \cite[Proposition 4.8]{BDN}). 

\begin{example}
 If $(\pi_M,N_M)$ is a $PN$-structure, then
 \begin{align*}
  \mathfrak{L}_{N_M^i}\mathrm{Gr}(\pi_M) = \mathrm{Gr}(N_M^i\pi_M).
 \end{align*}
If $(\omega_M,N_M)$ is an $\Omega N$-structure, then
\begin{align*}
 \mathfrak{R}_{N_M^i}\mathrm{Gr}(\omega_M) = \mathrm{Gr}(\omega_M N_M^i).
\end{align*}
\end{example}

\begin{corollary}
 In the setting of Lemma \ref{lem : MR for DN}, the shifts
\begin{align*}
 & \left( \mathfrak{L}_{N_M^n}(L_M), N_M \right),
 & \left( \mathfrak{R}_{N_M^n}(L_M), N_M \right),
\end{align*}
reduce to Dirac-Nijenhuis structures
\begin{align*}
 & \left( \mathfrak{L}_{N_Y^n}(L_Y), N_Y \right),
 & \left( \mathfrak{R}_{N_Y^n}(L_Y), N_Y \right),
\end{align*}
 wherever the latter are smooth. Moreover, any two such shifts concur weakly. 
\end{corollary}
\begin{proof}
As explained in \cite[Section 9]{FMT_Nijenhuis}, the left- and right shifts of Dirac-Nijenhuis structures are always involutive, as are the tangent and cotangent products of any two such shifts. Moreover,
\begin{align*}
 & \left( \mathfrak{L}_{N_M^n}(L_M), N_M \right),
 & \left( \mathfrak{R}_{N_M^n}(L_M), N_M \right)
\end{align*}
are Dirac-Nijenhuis structures wherever they are smooth. 
It follows from Lemma \ref{lem : MR for DN} that $(L_Y,N_Y)$, and therefore
\begin{align*}
 & \left( \mathfrak{L}_{N_Y^n}(L_Y), N_Y \right),
 & \left( \mathfrak{R}_{N_Y^n}(L_Y), N_Y \right),
\end{align*}
are Dirac-Nijenhuis wherever they are smooth, and since
\begin{align*}
 && p_!i^!(L_M) = L_Y, && N_X \sim_i N_M, && N_X \sim_p N_Y,
\end{align*}
we have that
\begin{align*}
 && p_!i^!\mathfrak{L}_{N_M^n}(L_M) = \mathfrak{L}_{N_Y^n}(L_Y), 
 && p_!i^!\mathfrak{R}_{N_M^n}(L_M) = \mathfrak{R}_{N_Y^n}(L_Y).
\end{align*}
\end{proof}

\begin{example}\normalfont
Let $\alpha \in \mathfrak{g}^*$ be a good value of a free Hamiltonian space
\begin{align*}
 G \curvearrowright (M,\mathrm{Gr}(\pi_0)) \stackrel{\mu}{\rmap} \mathfrak{g}^*,
\end{align*}
where $\pi_0$ is a symplectic Poisson structure, the inverse of a closed two-form $\omega_0 \in \Omega^2(M)$. Suppose $\omega_1$ is a $G_{\alpha}$-invariant, closed two-form, such that $N:=\pi_0^{\sharp}\omega_1^{\sharp} : TM \to TM$ is a Nijenhuis endomorphism which satisfies
\begin{align*}
 & N^{-1}(TX) = TX, & X:=\mu^{-1}(\alpha).
\end{align*}
Then
\begin{enumerate}[a)]
 \item $(\omega_0,N)$ is a $\Omega N$-structure,
 \item $(\pi_0,N)$ is a $PN$-structure,
 \item $E_G \subset TM|_X$ is a witness for (the graphs of)
 \begin{align*}
  & \omega_n:=\omega_0 N^n, & \pi_n:= N^n \pi_0,
 \end{align*}
where
\begin{align*}
E_G = \{\mathrm{a}(v)_x \ | \ v \in \mathfrak{g}, \ x \in X\}.
\end{align*}
\end{enumerate}
Indeed, by \cite[Proposition 20]{FMT_Concurrence}, $\mathrm{Gr}(\pi_0)$ and $\mathrm{Gr}(\omega_1)$ concur iff
\begin{align*}
 & \omega_1[u,v]^{\pi_0\omega_1} = [\omega_1(u),\omega_1(v)]^{\pi_0}, & u,v \in \mathfrak{X}(M).
\end{align*}
Because $\pi_0$ is symplectic, this happens exactly when $N=\pi_0\omega_1$ is Nijenhuis:
\begin{align*}
 N[u,v]^N = \pi_0\omega_1[u,v]^{\pi_0\omega_1} = \pi_0[\omega_1(u),\omega_1(v)]^{\pi_0} = [\pi_0\omega_1(u),\pi_0\omega_1(v)] = [N(u),N(v)].
\end{align*}
Thus $(\omega_0,N)$ is an $\Omega N$-structure, and therefore also $(\pi_0,N)$ is a $PN$-structure. By the moment map condition (Lemma \ref{lem : witness for good value}), $E_G$ is a witness for $\mathrm{Gr}(\pi_0)=\mathrm{Gr}(\omega_0)$. Because we assume that $N^{-1}(TX) = TX$, any Dirac structure $R$ of the form
\begin{align*}
 && R = \mathrm{Gr}(N^i\pi_0), && \text{or} && R = \mathrm{Gr}(\omega_0 N^j)
\end{align*}
satisfies
\begin{align}\label{eq : R cap E v R cap NX}
 R \cap E_G^{\perp} = R \cap N^*X^{\perp},
\end{align}
and because $\pi_i:=N^i\pi_0$ and $\omega_j:=\omega_0 N^j$ are $G_{\alpha}$-invariant, we also have that
\begin{align}\label{eq : F bracket L cap E perp}
 \left[ \Gamma(F), \Gamma(R \cap E^{\perp}) \right] \subset \Gamma(F) + \Gamma(R \cap E^{\perp}).
\end{align}
As in the proof of Lemma \ref{lem : witness for good value}, \eqref{eq : R cap E v R cap NX} and \eqref{eq : F bracket L cap E perp} imply that $E_G$ is a witness for $R$.
\end{example}

\begin{corollary}
  Let $L_M$ and $R_M$ be weakly concurring Dirac structures on $M$, and suppose
\begin{align*}
  \xymatrix{
 & M \ar[dl]_{p_L} \ar[dr]^{p_R} \ar[dd]^r & \\
 M_{L} \ar[dr]_{q_R} & & M_{R} \ar[dl]^{q_L}\\
 & M_{LR}
 }
\end{align*}
is the diagram of either Corollary \ref{cor : kernel diamond} or Corollary \ref{cor : Magri diamond}. If $N_M : TM \to TM$ is a Nijenhuis endomorphism, and
\begin{align*}
 & (L_M,N_M), & (R_M,N_M)
\end{align*}
are Dirac-Nijenhuis, then all vertices in the diagram above inherit Nijenhuis structures
\begin{align*}
 && N_{M_L}:TM_L \to TM_L,
 && N_{M_R}:TM_R \to TM_R,
 && N_{M_{LR}}:TM_{LR} \to TM_{LR},
\end{align*}
which are related by the diagram,
\begin{align*}
 && N_M \sim_{p_L} N_{M_L},
 && N_M \sim_{p_R} N_{M_R},
 && N_{M_L} \sim_{q_R} N_{M_{LR}},
 && N_{M_R} \sim_{q_L} N_{M_{LR}},
\end{align*}
and the induced Dirac structures on the vertices are Dirac-Nijenhuis for the corresponding endomorphism.
\end{corollary}
\begin{proof}
 In the context of Corollary \ref{cor : kernel diamond}, $(L_M,N_M)$ Dirac-Nijenhuis implies that
 \begin{align*}
  (\mathrm{Gr}(K(L_M)),N_M)
 \end{align*}
is Dirac-Nijenhuis by \cite[Proposition 3.13]{BDN}, and therefore $N_{M_L}:TM_L \to TM_L$ exists as in \cite[Corollary 4.5]{BDN}. Similarly, a Nijenhuis endomorphism $N_{M_R}:TM_R \to TM_R$ exists. As explained in \cite[Section 9]{FMT_Nijenhuis},
\begin{align*}
 (L_M \circledast R_M,N_M)
\end{align*}
is Dirac-Nijenhuis wherever smooth, and therefore the graph of
\begin{align*}
 K=K(L_M \circledast R_M)
\end{align*}
is Dirac-Nijenhuis for $N_M$. Therefore also a Nijenhuis endomorphism
\begin{align*}
 N_{M_{LR}}:TM_{LR} \to TM_{LR}
\end{align*}
exists. Thus the Dirac structures induced on $M_L$,$M_R$ and $M_{LR}$ are Dirac-Nijenhuis for the corresponding Nijenhuis endomorphisms. 
 
 In the context of Corollary \ref{cor : Magri diamond}, we observe as in \cite{FMT_Nijenhuis} that \eqref{eq : bracket of prolongation of a} is compatible with Dirac products, in the following sense: whenever $x_L,y_L \in \Gamma(L_M)$ and $x_R,y_R \in \Gamma(R_M)$ are such that
 \begin{align*}
  & \mathrm{pr}_T(x_L) = \mathrm{pr}_T(x_R),
  & \mathrm{pr}_{T^*}(y_L) = \mathrm{pr}_{T^*}(y_R),
 \end{align*}
 we write
 \begin{align*}
  & x_L \star x_R = x_L + \mathrm{pr}_{T^*}(x_R),
  & y_L \circledast y_R = y_L + \mathrm{pr}_T(y_R).
 \end{align*}
Then 
 \begin{align*}
  \{z,x_L \star x_R\}_{N_M} & = \{z,x_L\}_{N_M} \star \{z,x_R\}_{N_M},\\
  \{z,y_L \circledast y_R\}_{N_M} &= \{z,y_L\}_{N_M} \circledast \{z,y_R\}_{N_M}
 \end{align*}
 for all $z \in \Gamma(\mathbb{T}M)$. This implies that
 \begin{align*}
  & \mathscr{N}(L_M,R_M) = L_M \star \left( R_M \circledast \mathcal{R}_{-1}(L_M) \right), & \mathscr{N}(R_M,L_M) = R_M \star \left( L_M \circledast \mathcal{R}_{-1}(R_M) \right)
 \end{align*}
are Dirac-Nijenhuis for $N_M$ (wherever smooth), and therefore also
\begin{align*}
 && \mathscr{M}(L_M,R_M), 
 && \mathscr{M}(R_M,L_M), 
 && \mathscr{M}(L_M,R_M) \circledast \mathscr{M}(R_M,L_M)
\end{align*}
are Dirac-Nijenhuis for $N_M$ (wherever smooth). This implies that Nijenhuis endomorphisms
\begin{align*}
 && N_{M_L} : TM_L \to TM_L,
 && N_{M_R} : TM_R \to TM_R,
 && N_{M_{LR}} : TM_{LR} \to TM_{LR}
\end{align*}
exist, related by the submersions in the diagram, and all the Dirac structures induced on $M_L$,$M_R$ and $M_{LR}$ are Dirac-Nijenhuis for the corresponding Nijenhuis endomorphisms. 
\end{proof}

\subsection{Complex Dirac structures}

There is a version of Dirac geometry with complex coefficients, which goes as follows: the scalar extensions of the symmetric bilinear form and of the Dorfman bracket
\begin{align}\label{eq : C pairing and bracket}
 & \langle\cdot,\cdot\rangle : \mathbb{T}M \times \mathbb{T}M \to \mathbb{R}, & [\cdot,\cdot] : \Gamma(\mathbb{T}M) \times \Gamma(\mathbb{T}M) \to \Gamma(\mathbb{T}M) 
\end{align}
give rise to
\begin{align*}
 & \langle\cdot,\cdot\rangle_{\mathbb{C}} : \mathbb{T}_{\mathbb{C}}M \times \mathbb{T}_{\mathbb{C}}M \to \mathbb{C}, & [\cdot,\cdot]_{\mathbb{C}} : \Gamma(\mathbb{T}_{\mathbb{C}}M) \times \Gamma(\mathbb{T}_{\mathbb{C}}M) \to \Gamma(\mathbb{T}_{\mathbb{C}}M), 
\end{align*}
where $\mathbb{T}_{\mathbb{C}}M:=\mathbb{T}M \otimes \mathbb{C}$ stands for the complex generalized tangent bundle of $M$. A {\bf $\mathbb{C}$-Dirac} structure $L_M$ is a complex subbundle $L_M \subset \mathbb{T}_{\mathbb{C}}M$ which is Lagrangian and involutive:
\begin{align*}
 & L_M = L_M^{\perp}, & \left\langle \left[ \Gamma(L_M),\Gamma(L_M)\right]_{\mathbb{C}}, \Gamma(L_M)\right\rangle_{\mathbb{C}} = 0.
\end{align*}

\begin{example}\normalfont
If $L_M \subset \mathbb{T}M$ is a Dirac structure on $M$, its \emph{scalar extension} $L_M \otimes \mathbb{C}$ is a $\mathbb{C}$-Dirac structure on $M$. If $L_M \subset \mathbb{T}_{\mathbb{C}}M$ is a $\mathbb{C}$-Dirac structure on $M$, then so is its \emph{conjugate}
 \begin{align*}
  \overline{L_M} = \{ u + \xi \in \mathbb{T}_{\mathbb{C}}M \ | \ \overline{u}+\overline{\xi} \in L_M\}.
 \end{align*}
A $\mathbb{C}$-Dirac structure $L_M$ is the scalar extension of a Dirac structure if and only if it is \emph{real}, $L_M=\overline{L_M}$.
\end{example}

\begin{example}[Building blocks]\normalfont
 In analogy to the real case, there are notable building blocks for $\mathbb{C}$-Dirac structures on $M$:
 \begin{description}
  \item [involutive structures] for a complex vector subbundle $E \subset T_{\mathbb{C}}M$, $\mathrm{Gr}(E)=E\oplus E^{\circ}$ is a $\mathbb{C}$-Dirac structure exactly when $\Gamma(E)$ is involutive under \eqref{eq : C pairing and bracket};
  
  \item [closed complex forms] the graph of $\omega = \omega_1 + i\omega_2$, where $\omega_1,\omega_2 \in \Omega^2(M)$ is a $\mathbb{C}$-Dirac structure exactly when $\omega$ is closed; equivalently, if
 \begin{align*}
  \omega^{\sharp}[u,v]_{\mathbb{C}} = [\omega^{\sharp}(u),v]_{\mathbb{C}}+[u,\omega^{\sharp}(v)]_{\mathbb{C}}
 \end{align*}
 for all $u,v \in \Gamma(T_{\mathbb{C}}M)$;

\item [complex Poisson structures] the graph of $\pi = \pi_1 + i\pi_2$, where $\pi_1,\pi_2 \in \mathfrak{X}^2(M)$, is a $\mathbb{C}$-Dirac structure exactly when
 \begin{align*}
  \pi^{\sharp}[\xi,\eta]_{\mathbb{C}}^{\pi} = [\pi^{\sharp}(\xi),\pi^{\sharp}(\eta)]_{\mathbb{C}}
 \end{align*}
 for all $\xi,\eta \in \Gamma(T^*_{\mathbb{C}}M)$, where
 \begin{align*}
  [\xi,\eta]_{\mathbb{C}}^{\pi} = [\pi^{\sharp}(\xi),\eta]_{\mathbb{C}}+[\xi,\pi^{\sharp}(\eta)]_{\mathbb{C}}
 \end{align*}
 is the \emph{Koszul bracket} of $\pi$.
 \end{description}
\end{example}

\begin{remark}\normalfont
 Formal similarities notwithstanding, $\mathbb{C}$-Dirac structures are quite different from Dirac structures. That is mostly because involutive structures, like
 \begin{align*}
  E = \mathbb{C}\tfrac{\partial}{\partial z}  \subset T_{\mathbb{C}}\mathbb{C},
 \end{align*}
 do not in general partition the ambient space into leaves, unless said involutive structure is real, $E=\overline{E}$. See \cite{Wein07} for further details.
\end{remark}

The operations of Dirac products and stretching by an isotropic family are defined for $\mathbb{C}$-Lagrangian families, as they are in the real case. We describe the complex versions of the real reduction schemes: fix as before a triangle
\begin{align}\tag{$\mathfrak{T}$}
 M \stackrel{i}{\lmap} X \stackrel{p}{\rmap} Y
\end{align}
A \emph{complex} subbundle $E \subset T_{\mathbb{C}}M$ will be called {\bf adapted} if
\begin{align*}
 E \cap T_{\mathbb{C}}X = F_{\mathbb{C}},
\end{align*}
is the scalar extension of a foliation $F \subset TX$, in which case we write
\begin{align*}
 I_{\mathbb{C}} = F_{\mathbb{C}} \oplus N^*_{\mathbb{C}}X \subset \mathbb{T}_{\mathbb{C}}M|_X.
\end{align*}
An adapted subbundle is a {\bf complex witness} of a $\mathbb{C}$-Dirac structure $L_M \subset \mathbb{T}_{\mathbb{C}}M$ if:
\vspace{0.2cm}

\begin{minipage}[c]{0.9\textwidth}
\begin{mdframed}[backgroundcolor=olive!10]
\phantom{}
\begin{enumerate}[\text{CWit}1)]
  \item $L_M[I_{\mathbb{C}}]$ is a smooth subbundle of $\mathbb{T}_{\mathbb{C}}M$;
  \item $L_M[E] \cap (N_{\mathbb{C}}^*X)^{\perp}$ is involutive along $i$;
  \item $L_M \cap E^{\perp} \subset (N_{\mathbb{C}}^*X)^{\perp}$.
 \end{enumerate}
\end{mdframed} 
\end{minipage}
\vspace{0.2cm}

With this notion of \emph{complex witness} in place, we may state:

\begin{theorem}[Complex reduction schemes]\label{thm : complex MR reduction}
Let $L_M$ be a $\mathbb{C}$-Dirac structure on $M$. 
\begin{enumerate}[a)]
 \item $L_M[I_{\mathbb{C}}]$ is $\mathbb{C}$-Dirac exactly when $p_!i^!(L_M)=L_Y$ is a $\mathbb{C}$-Dirac structure on $Y$;
 \item $L_M[I_{\mathbb{C}}]$ is $\mathbb{C}$-Dirac if a complex witness $E$ exists for $L_M$;
 \item If $E$ is a complex witness for two weakly concurring $\mathbb{C}$-Dirac structures $L_M$ and $R_M$, then also $L_Y$ and $R_Y$ concur weakly;
 \item If $X=M$, $E$ is a complex witness for $L_M$ if and only if $L_M \circledast \mathrm{Gr}(E)$ is $\mathbb{C}$-Dirac.
 \item If $L_M=\mathrm{Gr}(\pi)$ corresponds to a complex Poisson structure, $E$ is a complex witness for $L_M$ if and only if $E^{\circ}$ is a Lie subalgebroid of $(T^*_{\mathbb{C}}M,[\cdot,\cdot]^{\pi})$.
\end{enumerate}
\end{theorem}
\begin{proof}
 This essentially consists of the observation that, because $E \cap T_{\mathbb{C}}X$ is the scalar extension of a foliation, none of the complications of $\mathbb{C}$-Dirac structures (vis \`{a} vis usual Dirac structures) intervene, and the arguments of Sections \ref{sec : Dirac reduction} and \ref{eq : Witnesses of concurrent reduction} carry over to the case of complex witnesses.
\end{proof}

\begin{example}\normalfont
 We illustrate how Magri's original recipe can be derived elegantly using the formalism of $\mathbb{C}$-Dirac structures. Let $\pi_L$ and $\pi_R$ be commuting Poisson structures, and define
 \begin{align*}
  & \pi \in \Gamma(\wedge^2T_{\mathbb{C}}M), & \pi=\pi_L+i\pi_R.
 \end{align*}
 Because $\pi_L$ and $\pi_R$ are Poisson and commute, the following are concurring $\mathbb{C}$-Dirac structures:
\begin{align*}
 & L=\mathrm{Gr}(\pi), & \overline{L}=\mathrm{Gr}(\overline{\pi}).
\end{align*}
Consider the complex Lagrangian family
\begin{align*}
 L \star \overline{L} = \{ \pi^{\sharp}(\xi)+\xi+\eta \ | \ \xi,\eta \in T^*_{\mathbb{C}}M, \ \pi^{\sharp}(\xi)=\overline{\pi}^{\sharp}(\eta)\}.
\end{align*}
Then $L \star \overline{L}$ is involutive. Because
\begin{align*}
 (L \star \overline{L}) \cap T_{\mathbb{C}}M  = \pi^{\sharp}(\ker(\pi^{\sharp}+\overline{\pi}^{\sharp}))
\end{align*}
 is the complexification of $E= \pi_R^{\sharp}(\ker(\pi_L^{\sharp}))$ \cite[Corollary 7.5]{Aguero_biv}, also $E$ is involutive; it is thus a foliation if $E$ has locally constant rank. This proves a). For b), suppose that the leaves of $E$ are the fibres of a surjective submersion $p:M \to Y$. Then a sufficient condition for $W:=E \otimes \mathbb{C}$ to be a complex witness for $L$ is that its annihilator
\begin{align*}
 W^{\circ} \subset (T^*_{\mathbb{C}}M,[\cdot,\cdot]^{\pi})
\end{align*}
be a complex Lie subalgebroid of the cotangent Lie algebroid of $\pi$. To see that this is so, note that
\begin{align*}
W^{\circ} =\pi^{\sharp}(\ker(\pi^{\sharp}+\overline{\pi}^{\sharp}))^{\circ} = (\pi^{\sharp})^{-1}(\mathrm{im}(\pi^{\sharp}+\overline{\pi}^{\sharp})),
\end{align*}
and suppose $\xi,\eta \in \Gamma(W^{\circ})$ are such that
\begin{align*}
 & \pi^{\sharp}(\xi) = (\pi^{\sharp}+\overline{\pi}^{\sharp})(\xi'), & \pi^{\sharp}(\eta) = (\pi^{\sharp}+\overline{\pi}^{\sharp})(\eta'),
\end{align*}
for some $\xi',\eta'\in \Gamma(T^*_{\mathbb{C}}M)$. Then
\begin{align*}
 \pi^{\sharp}[\xi,\eta]^{\pi} = [\pi^{\sharp}(\xi),\pi^{\sharp}(\eta)] = [(\pi^{\sharp}+\overline{\pi}^{\sharp})(\xi'),(\pi^{\sharp}+\overline{\pi}^{\sharp})(\eta')] = (\pi^{\sharp}+\overline{\pi}^{\sharp})[\xi',\eta']^{\pi+\overline{\pi}}
\end{align*}
shows that $[\xi,\eta]^{\pi} \in \Gamma(W^{\circ})$. By Theorem \ref{thm : complex MR reduction}, $p_!\mathrm{Gr}(\pi)$ is a $\mathbb{C}$-Dirac structure on $Y$, which is the graph of a complex Poisson structure $\nu=\nu_L+i\nu_R$ on $Y$ since
\begin{align*}
 p_!\mathrm{Gr}(\pi) \cap T_{\mathbb{C}}M & = \{p_*(u)+\xi \ | \ u=\pi^{\sharp}(\xi), \ \xi=0\} = 0.
\end{align*}
Note finally that
\begin{align*}
 & p_!\mathrm{Gr}(\pi_L) = \mathrm{Gr}(\nu_L), & p_!\mathrm{Gr}(\pi_R) = \mathrm{Gr}(\nu_R); 
\end{align*}
hence $E$ is a witness for $\pi_L$ and for $\pi_R$.
\end{example}

\subsection{The full picture of Magri's recipe}

In Magri's original recipe, rather than being performed globally, the reduction took place along a submanifold --- namely, the leaf of one of the Poisson structures. An analogous statement holds for the Dirac version of Magri's recipe, which generalizes the Poisson one. We state the ``left'' version only, the ``right'' one being obtained \emph{mutatis mutandis}:

\begin{proposition}\label{prop : X from L and R}
Let $L_M$ and $R_M$ be weakly concurring Dirac structures, for which
\begin{align*}
 & K:=K(L_M \circledast R_M)=(L_M \circledast R_M)\cap TM, & R_M \circledast \mathrm{Gr}(K)
\end{align*}
are vector bundles. Let $i:X \to M$ be a leaf of $R_M \circledast \mathrm{Gr}(K)$, and denote by
\begin{align*}
\mathscr{M}(L_M,R_M) = E_L \oplus E_L^{\circ}
\end{align*}
the Magri recipe of Definition \ref{def: Magri's recipe}. If $E_L \cap TX$ is a simple foliation, with leaf space $p:X \to Y$, and
\begin{align*}
 && E:=E_L|_X, && L_M \cap E, && R_M \cap E
\end{align*}
are smooth, then $E$ is simultaneously a witness for $L_M$ and for $R_M$.
\end{proposition}
\begin{proof}
Observe first that
\begin{align*}
 TX = K|_X + \Pi_R(K^{\circ}).
\end{align*}
Retain the notation of the proof of Theorem \ref{thm : new magri recipe}, and recall that
\begin{align*}
 E_L^{\circ} = \{ \xi \in K(L)^{\circ} \ | \ \Pi_L(\xi) \in \Pi_R(K^{\circ})+K\}.
\end{align*}
Because
\begin{align*}
 & L_M \cap E_L^{\perp} = K(L) + \mathcal{R}_{\Pi_L}(E_L^{\circ}),
 & R_M \cap E_L^{\perp} = K(R) + \mathcal{R}_{\Pi_R}(E_L^{\circ} \cap K^{\circ})
\end{align*}
by \eqref{eq : LEL} and \eqref{eq : REL}, we see that
\begin{align*}
 & L_M \cap E^{\perp} \subset N^*X^{\perp},
 & R_M \cap E^{\perp} \subset N^*X^{\perp}.
\end{align*}
Thus $E$ satisfies \hyperref[Wit3]{Wit3)} with respect to both $L_M$ and $R_M$, and this implies that
\begin{align*}
 L_M[I] & = L_M[E] \cap N^*X^{\perp} + N^*X = F + L_M \cap E^{\perp} + N^*X\\
 R_M[I] & = R_M[E] \cap N^*X^{\perp} + N^*X = F + R_M \cap E^{\perp} + N^*X.
\end{align*}
Because $L_M[I]$ and $R_M[I]$ are Lagrangian families, and $F$, $L_M \cap E^{\perp}$ and $R_M \cap E^{\perp}$ are smooth by hypothesis, it follows that $L_M[I]$ and $R_M[I]$ are smooth, and so $E$ satisfies \hyperref[Wit1]{Wit1)} with respect to both $L_M$ and $R_M$. Observe finally that, by Theorem \ref{thm : new magri recipe}, $L_M[E_L]$ and $R_M[E_L]$ are involutive, and therefore so are
\begin{align*}
 & L_M[E] \cap N^*V^{\perp} = L_M[E_L] \cap N^*V^{\perp},
 & R_M[E] \cap N^*V^{\perp} = R_M[E_L] \cap N^*V^{\perp}
\end{align*}
for all connected open set $V \subset X$ that $i$ embeds. Therefore $E$ also satisfies \hyperref[Wit2]{Wit2)} with respect to $L_M$ and $R_M$. 
\end{proof}

\begin{remark}\normalfont
In the setting of Proposition \ref{prop : X from L and R}, suppose $X \subset M$ is any submanifold which is saturated by leaves of $L_M$ and of $R_M$. Then the Magri recipe can be performed along $X$, under constant rank conditions similar to those of \emph{loc. cit.}.
\end{remark}

\begin{remark}\normalfont
Suppose a surjective submersion with connected fibres $r:M \to Y$ exists, whose vertical bundle is $K=K(L)+K(R)$. Then we have an equality
 \begin{align*}
  \mathscr{M}\left(r^!r_!(L),r^!r_!(R)\right) = r^!\mathscr{M}\left(r_!(L),r_!(R)\right)
 \end{align*}
Therefore the reduction under the Magri recipe (Theorem \ref{thm : new magri recipe}) of $r^!r_!(L)$ and $r^!r_!(R)$ coincides with the concatenation of the Reduction by kernels of Theorem \ref{thm : kernel}, followed by Magri's original recipe for the Poisson structures $r_!(L)$ and $r_!(R)$.
\end{remark}

\begin{example}\label{ex : kernel then magri}\normalfont
As shown in Example \ref{ex : kernel reduction}, the Dirac structures
\begin{align*}
  & L_M = \mathrm{Gr}\left(\mathrm{d}x_2 \wedge \mathrm{d}x_3 + \mathrm{d}x_4 \wedge \mathrm{d}x_5\right),
  & R_M = \mathrm{Gr}\left(\mathrm{d}x_1 \wedge \mathrm{d}x_2 + \mathrm{d}x_3 \wedge \mathrm{d}x_4 \right)
 \end{align*}
on $M=\mathbb{R}^5(x_1,x_2,x_3,x_4,x_5)$ both push forward under the submersion
\begin{align*}
 & r:\mathbb{R}^5 \to \mathbb{R}^3, & r(x_1,x_2,x_3,x_4,x_5) = (x_2,x_3,x_4)
\end{align*}
to the commuting Poisson structures
\begin{align*}
  && r_!(L_M) = \mathrm{Gr}\left( \nu_L \right), && \nu_L = \tfrac{\partial}{\partial x_3}\wedge \tfrac{\partial}{\partial x_2},
  && r_!(R_M) = \mathrm{Gr}\left( \nu_R \right), && \nu_R = \tfrac{\partial}{\partial x_4}\wedge \tfrac{\partial}{\partial x_3}
 \end{align*}
on $P:=M_{K(LR)} = \mathbb{R}^3(x_2,x_3,x_4)$. The original Magri recipe for $\nu_L$ and $\nu_R$ is
\begin{align*}
 E_{\mathrm{Gr}( \nu_L )} = \nu_L(\ker \nu_R) = \left\langle \tfrac{\partial}{\partial x_3} \right\rangle
\end{align*}
and therefore
\begin{align*}
  {p'_L}_!\mathrm{Gr}\left( \nu_L \right) = T^*P_{E_{\mathrm{Gr}( \nu_L )}} = {p'_L}_!\mathrm{Gr}\left( \nu_R \right),
 \end{align*}
where $P_{E_{\mathrm{Gr}( \nu_L )}} = \mathbb{R}^2(x_2,x_4)$. On the other hand, as shown in Example \ref{ex : Magri diamond}, the reductions of $L_M$ and $R_M$ under the Magri recipe
\begin{align*}
 E_{L_M} = K\left( L_M \star \left( R_M \circledast \mathcal{R}_{-1}(L_M) \right)\right) = \left\langle \tfrac{\partial}{\partial x_1}, \tfrac{\partial}{\partial x_3} \right\rangle
\end{align*}
are, respectively,
\begin{align*}
 & L_{M_{E_{L_M}}} = \mathrm{Gr}(\tfrac{\partial}{\partial x_5} \wedge \tfrac{\partial}{\partial x_4}), 
 & R_{M_{E_{L_M}}} = \mathrm{Gr}(\tfrac{\partial}{\partial x_5})
\end{align*}
on $M_{E_{L_M}} = \mathbb{R}^3(x_2,x_4,x_5)$. Therefore the new Magri recipe for $L_M$ and $R_M$ of Theorem \ref{thm : new magri recipe} does not consist of applying the original Magri recipe to the Poisson structures obtained by the common kernel reduction of Theorem \ref{thm : kernel}. 
\end{example}

\subsection{Comparison with the literature}We conclude the paper with a few remarks concerning the literature.

\begin{enumerate}[I)]
 \item The condition of involutivity along an injective immersion (Definition \ref{def : involutive along}) is a version of the condition involutivity along an embedded submanifold of \cite[Section 2]{Li-Bland} or \cite[Appendix B2]{BCH}. It seems new, and allows the discussion to happen at the generality of injective immersions, in the spirit of \cite[Chapter 10]{Ortega_Ratiu}.
 
 \item Even in the embedded case, Lemma \ref{lem : induced iff} refines the criterion in \cite[Theorem 3.1.1]{Courant} and \cite[Proposition 5.6]{Bursztyn} which ensures that the pullback $i^!(L)$ of a Dirac structure $L$ is again Dirac --- namely, that $L \cap N^*X$ be a vector bundle:
\begin{align*}
 && L \cap N^*X \ \text{VB} && \Longleftrightarrow && L \cap N^*X^{\perp} \ \text{VB} && \Longrightarrow && L[N^*X] = L \cap N^*X^{\perp} + N^*X \ \text{VB}.
\end{align*}

 \item The characterization of the split condition is a straightforward adaptation of \cite[Proposition 3.31]{Meinrenken_Poisson}; see also \cite[Lemma 2.15]{BFM}.

 \item Various reduction schemes for Dirac structures have been proposed in the literature. The majority of the literature \cite{CF,BGC,Zambon_Branes,CFZ,Sniatycki,FM_Dirac,CO,Balibanu} falls under the scope of what we call in the paper ``Dirac reduction'' --- namely, necessary conditions under which a Dirac structure can be transferred from $M$ to $Y$. Often such procedures are termed of ``Marsden-Ra\cb{t}iu kind'' by their authors. As discussed in the paper, in our terminology, we advocate that only \underline{concurring} reduction schemes should be regarded as of ``Marsden-Ra\cb{t}iu kind'', as that is an essential feature of the original Marsden-Ra\cb{t}iu reduction in Poisson geometry, as explained later in \cite{CMP,Costa_Marle,CFMP}. 
 
 \item In \cite{CF}, the authors consider the setting of a submersion $p:M \to B$, a Dirac structure $L_M$ on $M$ which can be pushed forward through $p$, and a submanifold $X \subset M$ for which $p:X \to Y=p(X)$ is still a submersion, and
 \begin{align*}
  p_!i^!p^!p_!(L_M) = L_Y
 \end{align*}
is a Dirac structure.

 \item The simplified approach in \cite{FZ} to (classical) Marsden-Ra\cb{t}iu reduction was very useful in formulating the general notion of witness, as was the observation in \cite{CFZ} that
 \begin{align*}
 i^!\left( \mathrm{Gr}(\pi_M)[E] \right) = p^!\mathrm{Gr}(\pi_Y).
\end{align*}
We observe, however that the ``extended Marsden-Ra\cb{t}iu reduction'' of \cite{FZ} moves in the direction of what we have called Dirac reduction, see e.g. \cite[Proposition 4.1]{FZ}.
 
 \item Dirac reduction (Theorem \ref{thm : Dirac reduction}) is discussed in \cite{BGC} and \cite{Zambon_Branes} in the more general situation in which the isotropic subbundle $I \subset \mathbb{T}M|_X$ is allowed to be a more general isotropic subbundle, and where a background $3$-form is allowed. We note that our results have a straightforward generalization to the case where $L_M$ is Dirac with respect to the Dorfman bracket \emph{twisted} by a closed 3-form $H_M \in \Omega^3(M)$,
 \begin{align*}
  [u+\xi,v+\eta]:=[u,v] + \mathscr{L}_u\eta - i_v\mathrm{d}\xi + i_{u \wedge v}H_M,
 \end{align*}
provided that the pullback of $H_M$ to $X$ be basic:
\begin{align}\label{eq : 3-form pulls to basic}
 & i^*(H_M) = p^*(H_Y), & H_Y \in \Omega^3(Y).
\end{align}
Under this condition, Theorems \ref{thm : Dirac reduction} and \ref{thm : reduction of weakly concurring concur weakly} remain true in this setting with background forms. Also the recipes of Theorems \ref{thm : kernel} and \ref{thm : new magri recipe} carry over to this setting if \eqref{eq : 3-form pulls to basic} is satisfied for the pertinent triangles. 
 
\item Our insistence that witnesses be subbundles of $TM|_X$ is directly tied to Theorem \ref{thm : reduction of weakly concurring concur weakly}: if one were to replace $\eqref{adapted subbundle}$ by a general isotropic subbundle of $\mathbb{T}M|_X$, two Dirac structures $L_M$ and $R_M$ witnessed by $E$ would have to satisfy additional conditions to ensure that $L_M \circledast R_M$ is also witnessed by $E$.

\item A different reduction scheme for \eqref{triangle} is considered in \cite{Balibanu}, in which the extra datum is given by a closed two-form $\omega \in \Omega^2(X)$\footnote{Their setting allows more generally for a background 3-form, and requires merely that $p:X \to Y$ have constant rank, but we ignore that for the sake of comparison of approaches.}. The Dirac reducibility condition is replaced by
\begin{align}\label{eq : BM reducible}\tag{BM}
 p_!\mathcal{R}_{-\omega}i^!(L_M) = L_Y
\end{align}
for a Dirac structure $L_Y$ on $Y$. The many beautiful examples in \emph{loc. cit.} attest the versatility and usefulness of this reduction scheme. However, extensions of Dirac reduction of this nature do not respect concurrence, as the example below illustrates.
\end{enumerate}

\begin{example}\normalfont
Let $M = X = Y = \mathbb{R}^3$, equipped with the Dirac structures
\begin{align*}
 & L = \left\langle \tfrac{\partial}{\partial x_1}, \tfrac{\partial}{\partial x_2}, \mathrm{d}x_3 \right\rangle,
 & R = \left\langle x_1\tfrac{\partial}{\partial x_3} + \mathrm{d}x_1, \mathrm{d}x_2, x_1\tfrac{\partial}{\partial x_1} - \mathrm{d}x_3 \right\rangle.
\end{align*}
Then
\begin{align*}
 L \circledast R = L
\end{align*}
shows that $L$ and $R$ concur. Consider the closed two-form
\begin{align*}
 \omega = \mathrm{d}x_1 \wedge \mathrm{d}x_2.
\end{align*}
Then
\begin{align*}
 & \mathcal{R}_{-\omega}(L) = \mathrm{Gr}\left( \tfrac{\partial}{\partial x_1} \wedge \tfrac{\partial}{\partial x_2} \right),
 & \mathcal{R}_{-\omega}(R) = \mathrm{Gr}\left( x_1\tfrac{\partial}{\partial x_1} \wedge \tfrac{\partial}{\partial x_3} \right)
\end{align*}
Because $i=p=\mathrm{id}_M$, this shows that both $L$ and $R$ satisfy (BM). However,
\begin{align*}
D:=\mathcal{R}_{-\omega}(L) \circledast \mathcal{R}_{-\omega}(R) = \left\langle a,b,c \right\rangle,
\end{align*}
where
\begin{align*}
 && a = \tfrac{\partial}{\partial x_1}-\mathrm{d}x_2, && b = \tfrac{\partial}{\partial x_2} + x_1\tfrac{\partial}{\partial x_3} + \mathrm{d}x_1, && c:= -x_1\tfrac{\partial}{\partial x_1} + \mathrm{d}x_3
\end{align*}
is not Dirac, since
\begin{align*}
\Upsilon_D(a,b,c) = \langle [a,b],c\rangle = 1.
\end{align*}
\end{example}

\end{document}